%% file: unstsf.tex
\documentclass[english,10pt]{amsart}

\usepackage{amsmath, amssymb, amsthm, enumerate}
\usepackage[all]{xy}
\usepackage[activeacute]{babel}
\usepackage[colorlinks=true, linkcolor=red, citecolor=blue]{hyperref}

\newtheorem{thm}[equation]{Theorem}
\newtheorem{lem}[equation]{Lemma}
\newtheorem{prop}[equation]{Proposition}
\newtheorem{cor}[equation]{Corollary}

\theoremstyle{remark}
\newtheorem{rmk}[equation]{Remark}

\theoremstyle{definition}
\newtheorem{defi}[equation]{Definition}
\newtheorem{exam}[equation]{Example}

\newcommand{\Hom}{\mathrm{Hom}}
\newcommand{\Gm}{\mathbb G _{m}}
\newcommand{\kdim}{\mathrm{dim}}
\newcommand{\blowup}{\mathcal B \ell}
\newcommand{\unsmot}{\mathcal M}
\newcommand{\unsmottran}{\mathcal M^{tr}}
\newcommand{\TspectraX}{Spt(\mathcal M)}
\newcommand{\unstablehomotopyX}{\mathcal{H}}
\newcommand{\stablehomotopyX}{\mathcal{SH}}
\newcommand{\qeffstablehomotopyX}{\Sigma _{T}^{q}\stablehomotopyX^{\mathit{eff}}}
\newcommand{\qorthogonalTspectraX}{L_{<q}\TspectraX}
\newcommand{\northogonalTspectraX}{L_{<n}\TspectraX}
\newcommand{\nplusoneorthogonalTspectraX}{L_{<n+1}\TspectraX}
\newcommand{\qplusoneorthogonalTspectraX}{L_{<q+1}\TspectraX}
\newcommand{\qorthogonalstablehomotopyX}{L_{<q}\stablehomotopyX}
\newcommand{\qorthogonalX}{\stablehomotopyX ^{\perp}(q)}

\newcommand{\nbiratunsmotX}{B_{n} \mathcal M}
\newcommand{\nplusonebiratunsmotX}{B_{n+1} \mathcal M}
\newcommand{\zerobiratunsmotX}{B_{0} \mathcal M}
\newcommand{\onebiratunsmotX}{B_{1} \mathcal M}
\newcommand{\homotnbiratunsmotX}{\unstablehomotopyX (B_{n})}
\newcommand{\homotnplusonebiratunsmotX}{\unstablehomotopyX (B_{n+1})}
\newcommand{\homotonebiratunsmotX}{\unstablehomotopyX (B_{1})}
\newcommand{\homotzerobiratunsmotX}{\unstablehomotopyX (B_{0})}
\newcommand{\nwbiratunsmotX}{WB_{n} \mathcal M}
\newcommand{\jwbiratunsmotX}{WB_{j} \mathcal M}
\newcommand{\rminonewbiratunsmotX}{WB_{r-1} \mathcal M}
\newcommand{\jminonewbiratunsmotX}{WB_{j-1} \mathcal M}
\newcommand{\nplusonewbiratunsmotX}{WB_{n+1} \mathcal M}
\newcommand{\nminonewbiratunsmotX}{WB_{n-1} \mathcal M}
\newcommand{\zerowbiratunsmotX}{WB_{0} \mathcal M}
\newcommand{\onewbiratunsmotX}{WB_{1} \mathcal M}
\newcommand{\homotnwbiratunsmotX}{\unstablehomotopyX (WB_{n})}
\newcommand{\homotnplusonewbiratunsmotX}{\unstablehomotopyX (WB_{n+1})}
\newcommand{\homotonewbiratunsmotX}{\unstablehomotopyX (WB_{1})}
\newcommand{\homotzerowbiratunsmotX}{\unstablehomotopyX (WB_{0})}
\newcommand{\nwbiratTspectraX}{WB_{n}Spt (\mathcal M)}

\newcommand{\nwbiratstablehomotopy}{\stablehomotopyX(WB_{n})}

\newcommand{\stepone}{\mathbf{C1}}
\newcommand{\steptwo}{\mathbf{C2}}

\newcommand{\topspc}{\mathbf{Top}}
\newcommand{\htop}{\mathbf{H}_{\mathrm{Top}}}
\newcommand{\infsusp}{\Sigma _{T}^{\infty}}
\newcommand{\infdel}{\Omega _{T}^{\infty}}

\numberwithin{equation}{section}

\begin{document}


\title{The Unstable Slice Filtration}


\author{Pablo Pelaez}
\address{Department of Mathematics, Rutgers University, U.S.A.}
\email{pablo.pelaez@rutgers.edu}


\subjclass[2010]{Primary 14F42}

\keywords{Birational Invariants, Motivic Homotopy Theory, Postnikov Tower,
	Slice Filtration, Unstable Slice Filtration}


\begin{abstract}
	The main goal of this paper is to construct an analogue of Voevodsky's slice
	filtration in the motivic unstable homotopy category.  The construction is
	done via birational invariants, this is motivated by
	the existence of an equivalence of categories between the orthogonal components
	for Voevodsky's slice filtration and the birational motivic stable homotopy categories
	constructed in \cite{Pelaez:2011fk}.  Another advantage of this approach is that
	the slices appear naturally as homotopy fibres (and not as in the stable setting,
	where they are defined as homotopy cofibres) which behave much
	better in the unstable setting.
\end{abstract}


\maketitle


\input{introd}
\input{sect1_unstsf}
\input{sect2_unstsf}
\input{sect3_unstsf}
\input{sect5_unstsf}
\input{sect6_unstsf}
\input{sect4_unstsf}


\section*{Acknowledgements}
	The author would like to thank Marc Levine for bringing to his attention the 
	connection between slices and birational cohomology theories, and Chuck Weibel 
	for pointing out that fibre sequences are a more reasonable way to define 
	Postnikov towers in the unstable setting.


\bibliography{biblio_unstsf}
\bibliographystyle{abbrv}

\end{document}

%% file: introd.tex
\begin{section}{Introduction}
		\label{Introd}
		
	Our main goal is to construct a tower in the motivic unstable homotopy category with similar properties 
	to Voevodsky's slice tower \cite{MR1977582}
	in the motivic stable homotopy category.  In this section we recall Voevodsky's definition for
	the slice filtration and explain why our construction in the unstable setting is a natural analogue.
		
\subsection*{Definitions and Notation}			
	In this paper $X$ will denote a Noetherian separated base scheme of finite Krull dimension,
	$Sch_{X}$ the category of schemes of finite type over $X$ and $Sm_{X}$ the full
	subcategory of $Sch_{X}$ consisting of smooth schemes over $X$ regarded
	as a site with the Nisnevich topology.	  All the maps between schemes will be considered over
	the base $X$.  Given $Y\in Sch_{X}$, all the closed subsets $Z$ of $Y$ will be considered
	as closed subschemes with the reduced structure.
	
	Let  $\unsmot$ be the category of pointed simplicial presheaves on $Sm_{X}$
	equipped with the motivic Quillen model structure \cite{MR0223432} constructed by Jardine 
	\cite[Thm.\,1.1]{MR1787949} which is Quillen equivalent to the one defined originally by
	Morel-Voevodsky \cite[Thm.\,1.2]{MR1787949}, taking the affine line $\mathbb A _{X}^{1}$ as interval 
	\cite[p.\,86 Thm.\,3.2]{MR1813224}.
	We will write $\unstablehomotopyX$ for its associated homotopy category.
	Given a map $f:Y\rightarrow W$ in $Sm_{X}$, we will abuse notation and denote
	by $f$ the induced map $f:Y_{+}\rightarrow W_{+}$ in $\mathcal M$ between the corresponding pointed
	simplicial presheaves represented by $Y$ and $W,$ respectively.
	
	We define $T$ in 
	$\unsmot$ as the pointed simplicial presheaf represented by 
	$S^{1}\wedge \mathbb G_{m}$, where $\mathbb G_{m}$ is the multiplicative group 
	$\mathbb A^{1}_{X}-\{ 0 \}$ pointed by $1$, and $S^{1}$ denotes the simplicial circle.
	Given an arbitrary integer $r\geq 1$, let  $S^{r}$ (resp. $\Gm ^{r}$) denote the
	iterated smash product of $S^{1}$ (resp. $\Gm$) with $r$-factors: $S^{1}\wedge \cdots \wedge S^{1}$
	(resp. $\Gm \wedge \cdots \wedge \Gm$);
	$S^{0}= \Gm ^{0}$ will be by definition equal to the pointed simplicial presheaf represented by the base scheme $X$.
	
	We will use the following notation in all the categories under consideration: $\ast$ will
	denote the terminal object, and $\cong$ will denote that a map is an isomorphism.
	
	Let $\TspectraX$ denote Jardine's category of symmetric $T$-spectra on 
	$\unsmot$ equipped with the motivic model structure defined in 
	\cite[theorem 4.15]{MR1787949} and $\stablehomotopyX$ denote its homotopy category, 
	which is triangulated.  We will follow Jardine's notation \cite[p. 506-507]{MR1787949}
	where $F_{n}$ denotes the left adjoint to the $n$-evaluation functor
		\[ \xymatrix@R=0.5pt{\TspectraX \ar[r]^-{ev_{n}}& \mathcal M \\
					(E^{m})_{m\geq 0} \ar@{|->}[r]& E^{n}.}
		\]
	Notice that $F_{0}(A)$ is just the usual infinite suspension spectrum $\Sigma _{T}^{\infty}A$.

	For every integer $q\in \mathbb Z$, we consider the following family of symmetric $T$-spectra
		\begin{align}
				\label{def.Cqeff}
			C^{q}_{\mathit{eff}}=\{ F_{n}(S^{r}\wedge \mathbb G _{m}^{s}\wedge U_{+}) 
				\mid n,r,s \geq 0; s-n\geq q; U\in Sm_{X}\}
		\end{align}
	where $U_{+}$ denotes the simplicial presheaf represented by $U$ with a disjoint base point.
	Let $\Sigma _{T}^{q}\stablehomotopyX^{\mathit{eff}}$ denote the smallest full triangulated subcategory of 
	$\stablehomotopyX$ which contains
	$C^{q}_{\mathit{eff}}$ and is closed under arbitrary coproducts.
	Voevodsky \cite{MR1977582} defines the slice filtration in $\stablehomotopyX$
	to be the following family of triangulated subcategories
		\[ \cdots \subseteq \Sigma _{T}^{q+1}\stablehomotopyX^{\mathit{eff}} \subseteq \Sigma _{T}^{q}\stablehomotopyX^{\mathit{eff}}
			\subseteq \Sigma _{T}^{q-1}\stablehomotopyX^{\mathit{eff}} \subseteq \cdots
		\]
	It follows from the work of Neeman \cite{MR1308405}, \cite{MR1812507} that the inclusion
		\[ i_{q}:\Sigma _{T}^{q}\stablehomotopyX^{\mathit{eff}}\rightarrow \stablehomotopyX
		\]
	has a right adjoint $r_{q}:\stablehomotopyX \rightarrow \Sigma _{T}^{q}
	\stablehomotopyX^{\mathit{eff}}$
	\cite[Prop. 3.1.12]{MR2807904}, and that the following functors
		\begin{align*}
				f_{q}: & \stablehomotopyX \rightarrow \stablehomotopyX \\
				s_{<q}: & \stablehomotopyX \rightarrow \stablehomotopyX \\
				s_{q}: & \stablehomotopyX \rightarrow \stablehomotopyX
		\end{align*}
	are triangulated, where $f_{q}$ is defined as the composition $i_{q}\circ r_{q}$; and $s_{<q}, s_{q}$ 
	are characterized by the fact that for every $E\in \stablehomotopyX$, there exist
	distinguished triangles in $\stablehomotopyX$ \cite[Thms. 3.1.16, 3.1.18]{MR2807904}:
		\begin{align}
				\label{orth.dist.triang}	
			\xymatrix{f_{q}E \ar[r]^-{\theta _{q}^{E}}& E \ar[r]^-{\pi _{<q}^{E}}& s_{<q}E \ar[r]& 
			S^{1}\wedge f_{q}E}\\
			\xymatrix{f_{q+1}E \ar[r]^-{\rho _{q}^{E}}& f_{q}E \ar[r]^-{\pi _{q}^{E}}& s_{q}E \ar[r]& 
			S^{1}\wedge f_{q+1}E .}
		\end{align}
	We will refer to $f_{q}E$ as the \emph{$(q-1)$-connective cover} of $E$, to $s_{<q}E$ as the 
	\emph{$q$-orthogonal component} of $E$, and to $s_{q}E$ as the \emph{$q$-slice} of $E$.
	It follows directly from the definition that $s_{<q+1}E, s_{q}E$ satisfy that
	for every symmetric $T$-spectrum $K$ in $\Sigma _{T}^{q+1}\stablehomotopyX^{\mathit{eff}}$:
		\[ \Hom _{\stablehomotopyX}(K,s_{<q+1}E)=\Hom _{\stablehomotopyX}(K,s_{q}E)=0.
		\]  
	Given a symmetric $T$-spectrum $E$, we will say that it is \emph{$q$-orthogonal}
	if for all $K\in \qeffstablehomotopyX$:
		\[	\Hom _{\stablehomotopyX}(K,E)=0.
		\]
	Let $\qorthogonalX$ denote the full subcategory of $\stablehomotopyX$ consisting
	of the $q$-orthogonal objects.

	Furthermore,
	the octahedral axiom for triangulated categories implies that the slices and the orthogonal
	components also fit in a natural distinguished triangle in $\stablehomotopyX$
	\cite[Prop. 3.1.19]{MR2807904}:
		\begin{equation}
				\label{eq.def.slices}
			\xymatrix{s_{q}E \ar[r] & s_{<q+1}E 
					\ar[r]& s_{<q}E \ar[r] & S^{1}\wedge s_{q}E .
							}
		\end{equation}
	
	The slice filtration admits an alternative definition in terms of Quillen model categories
	\cite{MR2576905, MR2807904},  namely:
	
\begin{thm}	
		\label{thm.modelstructures-slicefiltration}
	\begin{enumerate}
		\item	\label{thm.modelstructures-slicefiltration.a}
				There exists a Quillen model category $R_{C^{q}_{\mathit{eff}}}\TspectraX$ 
				which is defined as a right Bousfield localization of $\TspectraX$ with respect to the set of 
				objects $C^{q}_{\mathit{eff}}$ described in \eqref{def.Cqeff}, 
				such that its homotopy category $R_{C^{q}_{\mathit{eff}}}\stablehomotopyX$
				is triangulated and naturally equivalent to
				$\Sigma _{T}^{q}\stablehomotopyX^{\mathit{eff}}$.
				Moreover, the functor $f_{q}$ is canonically isomorphic to the following
				composition of triangulated functors:
					$$\xymatrix{\stablehomotopyX \ar[r]^-{R}& R_{C^{q}_{\mathit{eff}}}\stablehomotopyX \ar[r]^-{C_{q}}&
					 \stablehomotopyX}
					$$
				where $R$ is a fibrant replacement functor in $\TspectraX$, and $C_{q}$ a cofibrant replacement
				functor in $R_{C^{q}_{\mathit{eff}}}\TspectraX$.
		\item	\label{thm.modelstructures-slicefiltration.b}
				There exists a Quillen model category $\qorthogonalTspectraX$ 
				which is defined as a left Bousfield localization of $\TspectraX$ with respect to the set of 
				maps 
					\[ \{ F_{p}(S^{r}\wedge \Gm ^{s}\wedge U_{+})\rightarrow \ast \; | \;
							F_{p}(S^{r}\wedge \Gm ^{s}\wedge U_{+})\in C^{q}_{\mathit{eff}} \}			
					\] 
				such that its homotopy category $\qorthogonalstablehomotopyX$
				is triangulated and naturally equivalent to
				$\qorthogonalX$.
				Moreover, the functor $s_{<q}$ is canonically isomorphic to the following
				composition of triangulated functors:
					$$\xymatrix{\stablehomotopyX \ar[r]^-{Q}& L_{<q}\stablehomotopyX \ar[r]^-{W_{q}}&
					 \stablehomotopyX}
					$$
				where $Q$ is a cofibrant replacement functor in $\TspectraX$, and $W_{q}$ a fibrant replacement
				functor in $\qorthogonalTspectraX$.
		\item \label{thm.modelstructures-slicefiltration.c}
				There exists a Quillen model category $S^{q}\TspectraX$ which is defined as a right
				Bousfield localization of $\qplusoneorthogonalTspectraX$
				with respect to the set of objects
					\[	\{ F_{n}(S^{r}\wedge \mathbb G _{m}^{s}\wedge U_{+}) 
									\mid n,r,s \geq 0; s-n= q; U\in Sm_{X}\}
					\] 
				such that its homotopy category 
				$S^{q}\stablehomotopyX$ is triangulated
				and the identity functor 
					\[	id:R_{C^{q}_{\mathit{eff}}}\TspectraX \rightarrow S^{q}\TspectraX
					\]
				is a left Quillen functor.  Moreover, the functor $s_{q}$ is canonically isomorphic
				to the following composition of triangulated functors:
					$$\xymatrix{\stablehomotopyX \ar[r]^-{R} & R_{C^{q}_{\mathit{eff}}}\stablehomotopyX \ar[r]^-{C_{q}}& 
									S^{q}\stablehomotopyX \ar[r]^-{W_{q+1}}&
									R_{C^{q}_{\mathit{eff}}}\stablehomotopyX \ar[r]^-{C_{q}}& \stablehomotopyX}
					$$
	\end{enumerate}
\end{thm}
\begin{proof}
	(\ref{thm.modelstructures-slicefiltration.a}) and (\ref{thm.modelstructures-slicefiltration.c})
	follow directly from \cite[Thms. 3.3.9, 3.3.25]{MR2807904} and
	\cite[Thms. 3.3.50, 3.3.68]{MR2807904}, respectively.  Finally,
	(\ref{thm.modelstructures-slicefiltration.b}) follows from
	Proposition 3.2.27(3) together with Theorem 3.3.26, Proposition 3.3.30 and
	Theorem 3.3.45 in \cite{MR2807904}.
\end{proof}

	On the other hand, the model categories $\qorthogonalTspectraX$ admit an alternative description which is more
	suitable for an unstable interpretation.  For this, we
	fix an arbitrary integer  $n\geq 0$, and consider the following set of 
	open immersions with smooth closed complement of codimension at least $n+1$
		\begin{align}
				\label{nwbirat.locmaps}
			WB_{n}=\{ \iota _{U,Y}: &U_{+}\rightarrow Y_{+} \text{ open immersion } |\\
									& Y, Z=Y\backslash U \in Sm_{X}; Y \text{ irreducible};
									      (codim_{Y}Z)\geq n+1 \nonumber
					\} .
		\end{align}	
	Then we take the left Bousfield localization of the model category $\TspectraX$ with respect to the set of maps 
	\begin{align}	
			\label{infsusp.nwbirat.locmaps}
		sWB_{n}=\{ F_{p}(\Gm ^{b}\wedge \iota _{U,Y}): \iota _{U,Y}\in WB_{r};
						b, p, r\geq 0, r-p+b\geq n \}
	\end{align}
	which will be denoted by $\nwbiratTspectraX$.  The main result of \cite{Pelaez:2011fk} is the following
	comparison theorem:

\begin{thm}[see Theorem 3.6 in \cite{Pelaez:2011fk}]
		\label{thm.orthogonal=wbirationalTspectra}
	Let $n\in \mathbb Z$ be an arbitrary integer.  Then the identity functor
		\[ id: \nwbiratTspectraX \rightarrow \nplusoneorthogonalTspectraX
		\]
	is a Quillen equivalence.
\end{thm}

	Now, we proceed to describe how the unstable construction is carried out.
	In order to get an analogue for the slice filtration in the motivic unstable homotopy category it is necessary:
		\begin{description}
			\item[C1]	
					To remove from the picture as much as possible the role played by the suspensions 
					(with respect to $S^{1}$ and $\Gm$).
			\item[C2]	
					That the slices 
					appear naturally as homotopy fibres in the tower and not as homotopy cofibres, since the former behave much better in the
					unstable setting.
		\end{description}
		
	If we take ($\steptwo$) as a guiding principle, then the distinguished triangle (\ref{eq.def.slices}) and Theorem
	\ref{thm.modelstructures-slicefiltration}(\ref{thm.modelstructures-slicefiltration.b})-(\ref{thm.modelstructures-slicefiltration.c})
	lead us to consider analogues
	for the model structures $\northogonalTspectraX$ in the motivic unstable homotopy category
	$\unsmot$.  Once these model structures $\nwbiratunsmotX$ have been defined, we obtain a tower of left Quillen functors
		\[	\xymatrix@C=1.2pc{\unsmot \ar[r]^-{id} & \cdots \ar[r]^-{id} & \nplusonewbiratunsmotX \ar[r]^-{id}&
						\nwbiratunsmotX \ar[r]^-{id} &\cdots \ar[r]^-{id}& 
						\onewbiratunsmotX \ar[r]^-{id} & \zerowbiratunsmotX .}
		\]
	Thus, the derived adjunctions induce for every simplicial presheaf $A \in \unsmot$ a natural tower in $\unstablehomotopyX$
		\[	\xymatrix{A \ar[r] & \cdots \ar[r]^-{\nu _{n+1}^{A}} & t _{n+1}(A) \ar[r]^-{\nu _{n}^{A}}&
						t _{n}(A) \ar[r]^-{\nu _{n-1}^{A}} &\cdots \ar[r]^-{\nu _{2}^{A}}& 
						t _{2}(A) \ar[r]^-{\nu _{1}^{A}} & t _{1}(A) \ar[r]^-{\nu _{0}^{A}}& \ast}
		\]
	and the $n$-slice is then defined as the homotopy fibre of $\nu _{n}^{A}$, i.e. the following is a homotopy fibre sequence
	in $\unstablehomotopyX$
		\[	s_{n}(A)\rightarrow t _{n+1}(A) \rightarrow t _{n}(A).
		\]
	Now, the only problem that remains is the construction of the model structures 
	$\nwbiratunsmotX$.  For this, we consider ($\stepone$) and rely on
	Theorem \ref{thm.orthogonal=wbirationalTspectra}.  Hence, it is natural to define $\nwbiratunsmotX$
	as the left Bousfield localization of the model category $\unsmot$ with respect to the set of maps 
	$WB_{n}$ described in \eqref{nwbirat.locmaps}.
	
	We now describe the contents of the paper.
	In section \ref{Introd2}, we construct the weakly birational motivic homotopy categories 
	$\nwbiratunsmotX$ as well as a natural generalization that we call birational motivic
	homotopy categories, and prove a comparison theorem between them when the base scheme 
	is a perfect field (see Definitions \ref{def.localizationAmod-Birat}, \ref{def.localizationAmod-weaklyBirat}; 
	Proposition \ref{prop.genericsmoothness-Quillenequiv} and Theorem \ref{tower.Quillenfunctors}).  
	In section \ref{sect-1}, we define the unstable
	slices and describe some of their properties (see Proposition \ref{prop.slices=>homotopy} and Theorem \ref{thm.slice.tower.derived}).  
	In section \ref{sect-2}, we describe the
	unstable slice spectral sequence (see Theorem \ref{thm.extended.uns.slice.exact.couple}).
	In section \ref{sect-4}, we study the behavior of the unstable slice tower with respect to transfers in the
	sense of Voevodsky, and construct an unstable slice tower for simplicial presheaves with transfers
	(see Proposition \ref{uns.tower.tr} and Theorem \ref{thm.comp.tran}).
	In section \ref{sect-5}, we work over the complex numbers in order to study the behavior of the unstable slice
	tower with respect to the functor of complex points, and compare it with the classical Postnikov tower of
	algebraic topology
	(see Proposition \ref{prop.complexPtower}, Theorem \ref{thm.Post.real.comp} and Remark
	\ref{rmk.evenPosttow.alggeom}).
	In section \ref{sect-3}, we characterize the unstable slices, study the unstable slice tower for the infinite 
	loop space of a stable slice, and carry out some computations of unstable slices
	(see Theorem \ref{thm.class.uns.slices}; Propositions \ref{slice.orth.fibran.del}, \ref{prop.inf.susp.comp.fibr};
	Example \ref{infdel.not.comp.slice};
	and Theorems \ref{thm.vanishing.effective.cats}, \ref{thm.slices.proj.space}, \ref{thm.vanishing.normal.bundle},
	\ref{thm.slices-thom}, \ref{thm.blowups}).
	
\end{section}

%% file: sect1_unstsf.tex
\begin{section}{Birational and Weakly Birational Motivic Homotopy Categories}
		\label{Introd2}
		
	In this section, we construct the birational and weakly birational motivic homotopy categories.
	These are defined as left Bousfield localizations of $\unsmot$ with respect to maps which are induced
	by open immersions with a numerical condition in the codimension of the closed complement (which is
	assumed to be smooth in the weakly birational case).  The left Bousfield localizations are constructed following
	Hirschhorn's approach \cite{MR1944041}.  In order to be able to apply Hirschhorn's techniques, it is necessary
	to know that $\unsmot$ is \emph{cellular} \cite[Def. 12.1.1]{MR1944041} and \emph{proper} 
	\cite[Def. 13.1.1]{MR1944041}.  
	
\begin{thm}[{see \cite[Thm. 4.1.1]{MR1944041}}]
		\label{Hirsch-Bousloc}
	Let $\mathcal A$ be a Quillen model category which is cellular and proper.  Let $L$ be a set of maps in $\mathcal A$.
	Then, the left Bousfield localization of $\mathcal A$ with respect to $L$ exists.	
\end{thm}

	For details and definitions about Bousfield localization we refer the reader to Hirschhorn's book \cite{MR1944041}.
	The following result guarantees the existence of left Bousfield localizations for the unstable motivic homotopy
	category $\unsmot$.
	
\begin{thm}
		\label{Tspectra-cellular}
	The Quillen model category $\unsmot$ is:
		\begin{enumerate}  
			\item  \label{Tspectra-cellular.a}  \emph{cellular} (see 
							\cite[Cor. 1.6]{MR2197578} or \cite[Thm. 2.3.2]{MR2807904}).
			\item  \label{Tspectra-celllular.b}  \emph{proper} (see \cite[Thm. 1.1]{MR1787949}).
		\end{enumerate}
\end{thm}
		
	Now we describe the open immersions which will be used to construct the left Bousfield localizations
	of $\unsmot$.
	
\begin{defi}[{see \cite[section 7.5]{MR0338129}}]
		\label{def.codimension}
	Let $Y\in Sch_{X}$, and $Z$ a closed subscheme of $Y$.  The \emph{codimension}
	of $Z$ in $Y$, $codim_{Y}Z$ is the infimum (over the generic points $z_{i}$ of $Z$)
	of the dimensions of the local rings $\mathcal O _{Y, z_{i}}$. 
\end{defi}

	Since $X$ is Noetherian of finite Krull dimension and $Y$ is of finite type over $X$, $codim_{Y}Z$ is always finite.
		
\begin{defi}
		\label{def.localizing-maps}
	We fix an arbitrary integer  $n\geq 0$, and consider the following set of 
	open immersions which have a closed complement of codimension at least $n+1$
\begin{align*}
	B_{n}=\{ \iota _{U,Y}:& U_{+}\rightarrow Y_{+} \text{ open immersion } |\\														
											 & Y\in Sm_{X}; Y \text{ irreducible};
									     (codim_{Y}Y\backslash U)\geq n+1
	\}.
\end{align*}	
	The letter $B$ stands for birational.
\end{defi}

\begin{rmk}
		\label{rmk.Bn.set}
	The base scheme $X$ is Noetherian and of finite Krull dimension, therefore
	the smooth Nisnevich site $Sm_{X}$ is essentially small, i.e.
	there exists a small category $\mathcal S$ equivalent to $Sm_{X}$.  Hence, the collection
	of maps $\widetilde{B}_{n}$ in $\mathcal S$ which correspond to $B_{n}$ form a set.  Since the
	Bousfield localization of $\unsmot$ with respect to $\widetilde{B}_{n}$ is identical to the localization
	with respect to $B_{n}$, we will abuse terminology and simply say that $B_{n}$ is a set.
\end{rmk}

	Now we consider the left Bousfield localization of $\unsmot$ with respect to the set of
	open immersions $B_{n}$ described above.
	
\begin{defi}
		\label{def.localizationAmod-Birat}
	Let $n\geq 0$ be an arbitrary integer.
	\begin{enumerate}
		\item \label{def.localizationAmod-Birat.a}  Let $\nbiratunsmotX$ denote the
					left Bousfield localization of $\unsmot$ 
					with respect to the set of maps $B_{n}$.
		\item \label{def.localizationAmod-Birat.b}  Let $b^{(n)}_{u}$ denote its fibrant
				replacement functor and $\homotnbiratunsmotX$ its associated homotopy category.
	\end{enumerate}
	For $n\neq 0$ we will call $\homotnbiratunsmotX$ the \emph{codimension $n+1$-birational motivic homotopy
	 category}, and for $n=0$ we will call it the \emph{birational motivic homotopy category}.
\end{defi}

\begin{lem}
		\label{lemma.generators.becomeequiv.1}
	Let $n\geq 0$ be an arbitrary integer.  Then for every $a\geq 0$, the maps
		\[	S^{a}\wedge B_{n}=\{ S^{a}\wedge  
							\iota _{U,Y}: 
							\iota _{U,Y}\in B_{n} \}
		\]
	are weak equivalences in $\nbiratunsmotX$.
\end{lem}
\begin{proof}
	Let $\iota _{U,Y}
	\in B_{n}$.  We observe that $U_{+}, Y_{+}$ are cofibrant in $\nbiratunsmotX$,
	since they are cofibrant in $\unsmot$.   By construction,
	$\iota _{U,Y}$ is a weak equivalence  in 
	$\nbiratunsmotX$; and \cite[Thm. 4.1.1.(4)]{MR1944041} implies
	that $\nbiratunsmotX$ is a simplicial model category.
	Thus, it follows from
	Ken Brown's lemma (see \cite[Lem. 1.1.12]{MR1650134})
	that $S^{a}\wedge \iota _{U,Y}$ is also a weak
	 equivalence in $\nbiratunsmotX$ for every $a\geq 0$.
\end{proof}

\begin{prop}
		\label{prop.charac.fibrants1}
	Let $A$ be an arbitrary simplicial presheaf in $\unsmot$ and $n\geq 0$ an arbitrary integer.  If
	$A$ is fibrant in $\nbiratunsmotX$ then the following conditions hold:
		\begin{enumerate}
			\item	\label{prop.charac.fibrants1.a}  $A$ is fibrant in $\unsmot$.
			\item	\label{prop.charac.fibrants1.b}		For every $a\geq 0$
							 and every $\iota _{U,Y} \in B_{n}$,
							 the induced map
								\[ \xymatrix{\Hom _{\unstablehomotopyX}(S^{a}\wedge
										 Y_{+}, A) \ar[r]_-{\cong}^-{\iota _{U,Y}
										 ^{\ast}} &
										\Hom _{\unstablehomotopyX}(S^{a}
										\wedge U_{+}, A)}
								\]
							is an isomorphism.
		\end{enumerate}
\end{prop}
\begin{proof}
	Since the identity functor
		\[	id: \unsmot \rightarrow \nbiratunsmotX
		\]
	is a left Quillen functor, the conclusion follows from the derived
	adjunction 
		\[		(id, b_{u}^{(n)},\varphi ):\unstablehomotopyX \rightarrow
					 \homotnbiratunsmotX
		\]
	together with Lemma \ref{lemma.generators.becomeequiv.1}.
\end{proof}

Now we consider open immersions where the closed complement is smooth.

\begin{defi}
		\label{def.localizing-maps2}
	We fix an arbitrary integer  $n\geq 0$, and consider the following set of 
	open immersions with smooth closed complement	of codimension at least $n+1$
\begin{align*}
	WB_{n}=\{ \iota _{U,Y}: &U_{+}\rightarrow Y_{+} \text{ open immersion } |\\
									& Y, Z=Y\backslash U \in Sm_{X}; Y \text{ irreducible};
									      (codim_{Y}Z)\geq n+1
	\}.
\end{align*}	
\end{defi}

\begin{rmk}
		\label{rmk.WBn.set}
	As in Remark \ref{rmk.Bn.set}, the collection
	of maps $\widetilde{WB}_{n}$ in $\mathcal S$ which correspond to $WB_{n}$ form a set.  Since the
	Bousfield localization of $\unsmot$ with respect to $\widetilde{WB}_{n}$ is identical to the localization
	with respect to $WB_{n}$, we will abuse terminology and simply say that $WB_{n}$ is a set.
\end{rmk}

Now we describe the Bousfield localizations which will be used for the construction of the
unstable slice tower.

\begin{defi}
		\label{def.localizationAmod-weaklyBirat}
	Let $n\geq 0$ be an arbitrary integer.
	\begin{enumerate}
		\item \label{def.localizationAmod-weaklyBirat.a}  Let $\nwbiratunsmotX$ denote the
			left Bousfield localization of $\unsmot$ with respect to the set of maps 
			$WB_{n}$.
		\item \label{def.localizationAmod-weaklyBirat.b}  Let $wb_{u}^{(n)}$ denote its fibrant
			replacement functor and $\homotnwbiratunsmotX$ its associated homotopy category.
	\end{enumerate}
	For $n\neq 0$ we will call $\homotnwbiratunsmotX$ the \emph{codimension $n+1$-weakly birational motivic  homotopy
	category},
	and for $n=0$ we will call it the \emph{weakly birational motivic homotopy category}.
\end{defi}

\begin{prop}
		\label{prop.charac.fibrants2}
	Let $A$ be an arbitrary simplicial presheaf in $\unsmot$ and $n\geq 0$ an arbitrary integer.  If
	$A$ is fibrant in $\nwbiratunsmotX$ then the following conditions hold:
		\begin{enumerate}
			\item	\label{prop.charac.fibrants2.a}  $A$ is fibrant in $\unsmot$.
			\item	\label{prop.charac.fibrants2.b}		For every $a\geq 0$
							 and every $\iota _{U,Y} \in WB_{n}$,
							 the induced map
								\[ \xymatrix{\Hom _{\unstablehomotopyX}(S^{a}\wedge
										 Y_{+}, A) \ar[r]_-{\cong}^-{\iota _{U,Y}
										 ^{\ast}} &
										\Hom _{\unstablehomotopyX}(S^{a}\wedge 
										U_{+}, A)}
								\]
							is an isomorphism.
		\end{enumerate}
\end{prop}
\begin{proof}
	The proof is exactly the same as in Proposition 
	\ref{prop.charac.fibrants1}.
\end{proof}

\begin{prop}
		\label{prop.genericsmoothness-Quillenequiv}
	Let $n\geq 0$ be an arbitrary integer.
	Assume that the base scheme $X=\mathrm{Spec}\; k$, with $k$ a perfect field.
	Then the Quillen adjunction:
		$$(id,id,\varphi):\nwbiratunsmotX \rightarrow \nbiratunsmotX
		$$
	is a Quillen equivalence.
\end{prop}
\begin{proof}
	We have the following commutative diagram
		\[	\xymatrix{ & \unsmot \ar[dl]_-{id} \ar[dr]^-{id}& \\
							\nwbiratunsmotX \ar@{-->}[rr]_-{id}& & \nbiratunsmotX}
		\]
	where the solid arrows are left Quillen functors.  Clearly, 
	$WB_{n}\subseteq B_{n}$; thus, every $WB_{n}$-local equivalence is a $B_{n}$-local equivalence.
	Therefore,  the universal
	property of left Bousfield localizations 
	(see \cite[3.3.19.(1) and 3.1.1.(1)]{MR1944041}) implies that the horizontal 
	arrow is also a left Quillen functor.
	
	The universal property for left Bousfield localizations also implies that
	it is enough to show that all the maps $\iota _{U,Y}$ in $B_{n}$
	become weak equivalences in $\nwbiratunsmotX$.  
	We proceed by induction
	on the dimension of $Z=Y\backslash U$.  If $\kdim \; Z=0$, then
	$Z\in Sm_{X} $ since $k$ is a perfect field (and we are considering
	$Z$ with the reduced scheme structure), hence
	$\iota _{U,Y}\in WB_{n}$ and then a
	weak equivalence in $\nwbiratunsmotX$.
	
	If $\kdim \; Z>0$, then
	we consider the singular locus $Z_{s}$ of $Z$ over $X$.  We have that
	$\kdim \; Z_{s}<\kdim \; Z$ since $k$ is a perfect field.  Therefore, by
	induction on the dimension $\iota _{V,Y}$
	is a weak equivalence in $\nwbiratunsmotX$, where $V=Y\backslash
	Z_{s}$.  On the other hand, $\iota _{U,V}$
	is also a weak equivalence in $\nwbiratunsmotX$, since 
	it is in $B_{n}$ and its closed
	complement $V\backslash U=Z\backslash Z_{s}$ is smooth
	over $X$, by construction of $Z_{s}$.
	
	But $\iota _{U,Y}= 
	\iota _{V,Y}\circ \iota _{U,V}$; so
	by the two out of three property for weak equivalences we
	conclude that $\iota _{U,Y}$ is a weak equivalence
	in $\nwbiratunsmotX$.
\end{proof}

\begin{thm}
		\label{tower.Quillenfunctors}
	Let $n\geq 0$ be an arbitrary integer.  Then,
		\begin{enumerate}
			\item	\label{tower.Quillenfunctors.a}		The following diagram is a tower of left Quillen functors
						\[	\xymatrix{&&& \ar[dl]^-{id} \ar[dll]_-{id} \unsmot \ar[dr]_-{id} \ar[drr]^-{id} && \\
										\cdots \ar[r]_-{id} &  \nplusonebiratunsmotX \ar[r]_-{id}& \nbiratunsmotX \ar[r]_-{id}& 
										\cdots \ar[r]_-{id} & \onebiratunsmotX \ar[r]_-{id}& 
										\zerobiratunsmotX
										}						
						\]
					and
						\[	\xymatrix{&&& \ar@<-1ex>[dddl]_-{id} \ar@<-3ex>[dddll]_-{id} \unstablehomotopyX 
										 \ar@<1ex>[dddr]^-{id} \ar@<3ex>[dddrr]^-{id} && \\
										&&&&& \\
										&&&&& \\
										\cdots \ar@<1ex>[r]^-{id}  & \ar@<1ex>[l]^-{b_{u}^{(n+1)}}  
										\homotnplusonebiratunsmotX \ar@<1ex>[r]^-{id} \ar@<1ex>[uuurr]|{b_{u}^{(n+1)}}
										& \ar@<1ex>[l]^-{b_{u}^{(n)}}
										\homotnbiratunsmotX \ar@<1ex>[r]^-{id} \ar@<-1ex>[uuur]|{b_{u}^{(n)}} & \ar@<1ex>[l]^-{b_{u}^{(n-1)}}
										\cdots \ar@<1ex>[r]^-{id} & \ar@<1ex>[l]^-{b_{u}^{(1)}} \homotonebiratunsmotX 
										\ar@<1ex>[r]^-{id} \ar@<1ex>[uuul]|{b_{u}^{(1)}} & 
										\ar@<1ex>[l]^-{b_{u}^{(0)}} \homotzerobiratunsmotX \ar@<-1ex>[uuull]|{b_{u}^{(0)}}
										}						
						\]
					is the corresponding tower of associated homotopy categories.
			\item	\label{tower.Quillenfunctors.b}		The following diagram is a tower of left Quillen functors
						\[	\xymatrix{&&& \ar[dl]^-{id} \ar[dll]_-{id} \unsmot \ar[dr]_-{id} \ar[drr]^-{id} && \\
										\cdots \ar[r]_-{id} &  \nplusonewbiratunsmotX \ar[r]_-{id}& \nwbiratunsmotX \ar[r]_-{id}& 
										\cdots \ar[r]_-{id} & \onewbiratunsmotX \ar[r]_-{id}& 
										\zerowbiratunsmotX
										}						
						\]
					and
						\[	\xymatrix{&&& \ar@<-1ex>[dddl]_-{id} \ar@<-3ex>[dddll]_-{id} \unstablehomotopyX 
										 \ar@<1ex>[dddr]^-{id} \ar@<3ex>[dddrr]^-{id} && \\
										&&&&& \\
										&&&&& \\
										\cdots \ar@<1ex>[r]^-{id}  & \ar@<1ex>[l]^-{wb_{u}^{(n+1)}}  
										\homotnplusonewbiratunsmotX \ar@<1ex>[r]^-{id} \ar@<1ex>[uuurr]|{wb_{u}^{(n+1)}}
										& \ar@<1ex>[l]^-{wb_{u}^{(n)}}
										\homotnwbiratunsmotX \ar@<1ex>[r]^-{id} \ar@<-1ex>[uuur]|{wb_{u}^{(n)}} & 
										\ar@<1ex>[l]^-{wb_{u}^{(n-1)}}
										\cdots \ar@<1ex>[r]^-{id} & \ar@<1ex>[l]^-{wb_{u}^{(1)}} \homotonewbiratunsmotX 
										\ar@<1ex>[r]^-{id} \ar@<1ex>[uuul]|{wb_{u}^{(1)}} & 
										\ar@<1ex>[l]^-{wb_{u}^{(0)}} \homotzerowbiratunsmotX \ar@<-1ex>[uuull]|{wb_{u}^{(0)}}
										}						
						\]
					is the corresponding tower of associated homotopy categories.
		\end{enumerate}
\end{thm}
\begin{proof}
	Since the model categories $\nbiratunsmotX$ are defined as left Bousfield localizations of $\unsmot$
	with respect to the set of maps $B_{n}$ (see Definition \ref{def.localizing-maps}),
	it follows that $id:\unsmot \rightarrow \nbiratunsmotX$ is a left Quillen functor.  On the other hand,
	we observe that $B_{n+1}\subseteq B_{n}$; hence every $B_{n+1}$-local equivalence is also a 
	$B_{n}$-local equivalence, and the universal property of left Bousfield localizations 
	(see \cite[3.3.19.(1) and 3.1.1.(1)]{MR1944041}) implies that
	$id:\nplusonebiratunsmotX \rightarrow \nbiratunsmotX$ is a left Quillen functor.  Finally, we observe
	that the identity functor is a cofibrant replacement functor in $\unsmot$ since we are using the injective model
	structure (see \cite[Thms. 2.1.1, 2.3.1 and 2.3.2 ]{MR2807904}).
	This proves (\ref{tower.Quillenfunctors.a}).  The proof of (\ref{tower.Quillenfunctors.b}) is exactly the same.
\end{proof}

\end{section}

%% file: sect2_unstsf.tex
\begin{section}{The Unstable Slice Filtration}
		\label{sect-1}
		
	In this section, we construct the unstable slice functors $s_{n}$ (for $n\geq 0$)
	as well as their respective liftings $\tilde{s}_{n}$ to the level of Quillen model categories.
	
\begin{lem}
		\label{lemma.weakequivalences.under.tower}
	Let $0\leq j\leq n$ be arbitrary integers, and $f:A\rightarrow B$ a map in $\unsmot$ which
	is a weak equivalence in $\nwbiratunsmotX$.  Then $f$ is also a weak equivalence in
	$\jwbiratunsmotX$.
\end{lem}
\begin{proof}
	Theorem \ref{tower.Quillenfunctors} implies that the identity functor 
	$id:\nwbiratunsmotX \rightarrow \jwbiratunsmotX$ is
	a left Quillen functor, and by construction both model categories are left Bousfield localizations of
	$\unsmot$.  Hence by the universal property of left Bousfield localizations 
	(see \cite[3.3.19.(1) and 3.1.1.(1)]{MR1944041}) it follows that
	$f$ is also a weak equivalence in $\jwbiratunsmotX$.
\end{proof}

\begin{defi}
		\label{def.slicetower}
	Let $n\geq 1$ be an arbitrary integer.  We define inductively the following functors
		\[	\tau _{n}:\unsmot \rightarrow \unsmot
		\]
	together with natural transformations $\tau _{\leq n}:id \rightarrow \tau _{n}$ and
	$\nu _{n}:\tau _{n+1}\rightarrow \tau _{n}$  as follows:
		\begin{enumerate}
			\item \label{step1.slicetower}	$\tau _{1}$ is just $wb_{u}^{(0)}$, i.e. the fibrant replacement functor
						in $\zerowbiratunsmotX$, and $\tau _{\leq 1}$ is the natural transformation induced by the
						functorial factorization in $\zerowbiratunsmotX$
						of a map into a trivial cofibration followed by a fibration.  Thus, every map
						$f:A\rightarrow B$ in $\unsmot$, fits in the following commutative diagram
							\[	\xymatrix{A \ar[r]^-{\tau _{\leq 1}^{A}} \ar[d]_-{f} & 
							\tau _{1}A \ar[r] \ar[d]^-{\tau _{1}(f)} & \ast \\
							B \ar[r]_-{\tau _{\leq 1}^{B}} & \tau _{1}B \ar[r] & \ast }
							\]
			\item	\label{step2.slicetower}	Assume that $\tau _{n}$ and $\tau _{\leq n}$ are already defined.  Then,
						given an arbitrary simplicial presheaf $A\in \unsmot$, we factor the natural map
							\[	\tau _{\leq n}^{A}:A\rightarrow \tau _{n}A
							\]
						as a trivial cofibration (in $\nwbiratunsmotX$) followed by a fibration (in $\nwbiratunsmotX$)
							\[	\xymatrix{& \ar[dl]_-{\tau _{\leq n+1}^{A}} A \ar[dr]^-{\tau _{\leq n}^{A}}& \\
							\tau _{n+1}A \ar[rr]_-{\nu _{n}^{A}} && \tau _{n}A }
							\]
		\end{enumerate}
\end{defi}

\begin{rmk}
		\label{rmk.def.slice.tower}
	The functoriality of the factorizations in $\nwbiratunsmotX$, implies that $\tau _{n+1}$ is
	a functor in $\unsmot$ and that $\tau _{\leq n+1}: id\rightarrow \tau _{n+1}$, 
	$\nu _{n}:\tau _{n+1}\rightarrow \tau _{n}$ are natural transformations.
\end{rmk}

\begin{prop}
		\label{prop.properties.slice.tower}
	Let $n\geq 0$ be an arbitrary integer, and $A\in \unsmot$ be an arbitrary simplicial presheaf.  Then:
		\begin{enumerate}
			\item	\label{prop.properties.slice.tower.a}	The natural map 
						\[	\tau _{\leq n+1}^{A}:A\rightarrow \tau _{n+1}A
						\]
					is a trivial cofibration in $\nwbiratunsmotX$.
			\item	\label{prop.properties.slice.tower.b}	The natural map (for $n\geq 1$)
						\[	\nu _{n}^{A}:\tau _{n+1}A\rightarrow \tau _{n}A
						\]
					is a fibration in $\nwbiratunsmotX$, and
						\[	\tau _{1}A\rightarrow \ast
						\]
					is a fibration in $\zerowbiratunsmotX$.
			\item	\label{prop.properties.slice.tower.c}	The composition
						\[	\xymatrix{\tau _{n+1}A \ar[r]^-{\nu _{n}^{A}} & \tau _{n}A \ar[r]^-{\nu _{n-1}^{A}}& 
											\cdots \ar[r]& \tau _{2}A \ar[r]^{\nu _{1}^{A}}& \tau _{1}A \ar[r]& \ast}
						\]
					is a fibration in $\unsmot$ and $\nwbiratunsmotX$.
		\end{enumerate}
\end{prop}
\begin{proof}
	(\ref{prop.properties.slice.tower.a}) and (\ref{prop.properties.slice.tower.b}) follow directly from the
	construction of the functors $\tau _{n}$, and the natural transformations $\tau _{\leq n}$, $\nu _{n}$
	in Definition \ref{def.slicetower}.  
	
	(\ref{prop.properties.slice.tower.c}):  By construction $\nwbiratunsmotX$ is a left Bousfield localization of $\unsmot$.
	Thus, the identity functor $id:\nwbiratunsmotX \rightarrow \unsmot$ is a right 
	Quillen functor, so it is enough to show the result for $\nwbiratunsmotX$.
	For this, it suffices to show that $\tau _{1}A$ is fibrant in $\nwbiratunsmotX$ 
	and that all the maps $\nu _{n}^{A}$ are fibrations in $\nwbiratunsmotX$. But this follows from (\ref{prop.properties.slice.tower.b})
	above together with Theorem \ref{tower.Quillenfunctors}.
\end{proof}

\begin{cor}
	\label{cor.slice.tower.derived.a}
Let $n\geq 1$ be an arbitrary integer, and $A\in \unsmot$ be an arbitrary simplicial presheaf.  Then, the natural map 
	\[	\nu _{n}^{A}:\tau _{n+1}A\rightarrow \tau _{n}A
	\]
is a weak equivalence in $\nminonewbiratunsmotX$.
\end{cor}
\begin{proof}
By construction, the following diagram commutes
	\begin{align*} \xymatrix{A \ar[r]^-{\tau _{\leq n}^{A}} \ar[d]_-{\tau _{\leq n+1}^{A}}& \tau _{n}A \\
									\tau _{n+1}A \ar[ur]_-{\nu _{n}^{A}}&}
	\end{align*}
Combining Lemma \ref{lemma.weakequivalences.under.tower} and Proposition \ref{prop.properties.slice.tower}\eqref{prop.properties.slice.tower.a}
we deduce that the maps $\tau _{\leq n+1}^{A}, \tau _{\leq n}^{A}$ are weak
equivalences in $\nminonewbiratunsmotX$.  Thus, by
the two out of three property for weak equivalences we conclude that $\nu _{n}A$ is a weak equivalence in
$\nminonewbiratunsmotX$.  
\end{proof}

\begin{defi}
		\label{def.constr.slice}
	\begin{enumerate}
		\item	\label{def.constr.slice.a}	Let $n\geq 1$ be an arbitrary integer, and
		$A$ be an arbitrary simplicial presheaf  in $\unsmot$.  
		Consider the natural map in $\unsmot$ (see Definition \ref{def.slicetower}).
			\[
				\nu _{n}^{A}:\tau _{n+1}A\rightarrow \tau _{n}A 
			\]
		Let $\tilde{s} _{n}A$ denote the pullback in $\unsmot$ of $\nu _{n}^{A}$ along the base point
		 of $\tau _{n}A$.
			\begin{align}
					\label{def.equation.constr.slice.a}
				\begin{array}{c}
					\xymatrix{\tilde{s} _{n}A \ar[d] 
								\ar[r]^-{i_{n}^{A}} & \tau _{n+1}A \ar[d]^-{\nu _{n}^{A}}\\
								\ast \ar[r] & \tau _{n}A}
				\end{array}
			\end{align}
		\item	\label{def.constr.slice.b}	If $n=0$, let $\tilde{s}	_{0}=\tau _{1}$.
	\end{enumerate}
\end{defi}

\begin{prop}
		\label{prop.slice=>functor}
	Let $n\geq 1$ be an arbitrary integer.  Then
	the construction 
	given in Definition \ref{def.constr.slice} defines a functor
		\[ \tilde{s} _{n} :\unsmot \rightarrow \unsmot
		\]
	and a natural transformation $i_{n}:\tilde{s} _{n}\rightarrow \tau _{n+1}$.
\end{prop}
\begin{proof}
	It follows immediately from the fact that $\nu _{n}$ is a natural transformation
	between functors in $\unsmot$, and the universal property of pullbacks in $\unsmot$.
\end{proof}

\begin{lem}
		\label{lem.slices.preserve.wequivs}
	Let $n\geq 1$ be an arbitrary integer.  Then:
		\begin{enumerate}
			\item	\label{lem.slices.preserve.wequivs.a}	The functor (see Definition \ref{def.slicetower})
						\[ \tau _{n} :\unsmot \rightarrow \unsmot
						\]
					maps weak equivalences in $\nminonewbiratunsmotX$ to weak equivalences in $\unsmot$.
			\item	\label{lem.slices.preserve.wequivs.b}	The functor 
			(see Proposition \ref{prop.slice=>functor})
						\[ \tilde{s}_{n} :\unsmot \rightarrow \unsmot
						\]
					maps weak equivalences in $\nwbiratunsmotX$ to weak equivalences in $\unsmot$.
		\end{enumerate}
\end{lem}
\begin{proof}
	  (\ref{cor.slices.preserve.wequivs.a}):  Let $f:A\rightarrow B$ be a weak equivalence in 
	  $\nminonewbiratunsmotX$.  We consider the following commutative diagram in
	$\nminonewbiratunsmotX$
		\[	\xymatrix{A \ar[rr]^-{f} \ar[d]_-{\tau ^{A}_{\leq n}} && B \ar[d]^-{\tau ^{B}_{\leq n}}\\
								\tau _{n}A \ar[rr]_-{\tau _{n}(f)} && \tau _{n}B}
		\]
	by hypothesis $f$ is a weak equivalence in $\nminonewbiratunsmotX$.
	On the other hand, it follows from Proposition
	\ref{prop.properties.slice.tower}\eqref{prop.properties.slice.tower.a} that 
	the vertical maps are also weak equivalences
	in $\nminonewbiratunsmotX$, and the two out of three property for weak equivalences implies
	that $\tau _{n}(f)$ is also a weak equivalence in $\nminonewbiratunsmotX$.  Finally, by Proposition
	\ref{prop.properties.slice.tower}(\ref{prop.properties.slice.tower.c})
	$\tau _{n}A$, $\tau _{n}B$ are both fibrant in
	$\nminonewbiratunsmotX$, thus \cite[Thm. 3.2.13(1) and 
	Prop. 3.4.1(1)]{MR1944041} 
	implies that $\tau _{n}(f)$ is
	a weak equivalence in $\unsmot$.
	
	(\ref{cor.slices.preserve.wequivs.b}):	Let $f:A\rightarrow B$ be a weak equivalence in 
	$\nwbiratunsmotX$.
	By construction (see Definition \ref{def.constr.slice}
	and Proposition \ref{prop.properties.slice.tower}(\ref{prop.properties.slice.tower.b})), the following are fibre 
	sequences in $\nwbiratunsmotX$
		\[	\xymatrix@R=0.5pt{\tilde{s}_{n} A \ar[r]& 
							\tau _{n+1}A \ar[r]^-{\nu _{n}^{A}} & 
							\tau _{n}A \\
							\tilde{s}_{n} B \ar[r]& 
							\tau _{n+1}B \ar[r]_-{\nu _{n}^{B}} & 
							\tau _{n}B}
		\]
	Now, Theorem \ref{tower.Quillenfunctors} implies that they are also fibre sequences 
	in $\unsmot$.  On the other hand,
	by naturality, $f$ induces a map between these fibre sequences
		\[	\xymatrix{\tilde{s}_{n} A \ar[r]
							\ar[d]^-{\tilde s_{n}(f)} & \tau _{n+1}A \ar[r]^-{\nu _{n}^{A}} 
							\ar[d]^-{\tau _{n+1}(f)}& \tau _{n}A \ar[d]^-{\tau _{n}(f)} \\
							\tilde{s}_{n} B \ar[r]& 
							\tau _{n+1}B \ar[r]_-{\nu _{n}^{A}} & 
							\tau _{n}B}
		\]
	Combining (\ref{lem.slices.preserve.wequivs.a}) above with Lemma \ref{lemma.weakequivalences.under.tower},
	we deduce that $\tau _{n+1}(f)$, $\tau _{n}(f)$ are weak equivalences in $\unsmot$.  Thus, the result
	follows from \cite[I.3 Prop. 5(iii)]{MR0223432}.
\end{proof}

\begin{cor}
		\label{cor.slices.preserve.wequivs}
	Let $n\geq 1$ be an arbitrary integer.  Then:
		\begin{enumerate}
			\item	\label{cor.slices.preserve.wequivs.a}	The functor (see Definition \ref{def.slicetower})
						\[ \tau _{n} :\unsmot \rightarrow \unsmot
						\]
					maps weak equivalences in $\unsmot$ to weak equivalences in $\unsmot$.
			\item	\label{cor.slices.preserve.wequivs.b}	The functor 
			(see Proposition \ref{prop.slice=>functor})
						\[ \tilde{s}_{n} :\unsmot \rightarrow \unsmot
						\]
					maps weak equivalences in $\unsmot$ to weak equivalences in $\unsmot$.
		\end{enumerate}
\end{cor}
\begin{proof}
	Let $f:A\rightarrow B$ be a weak equivalence in $\unsmot$.  By lemma \ref{lem.slices.preserve.wequivs}
	it suffices to show that for every $n\geq 0$, $f$ is also a weak equivalence in $\nwbiratunsmotX$.
	But this follows from
	\cite[Prop. 3.3.3(1)(a)]{MR1944041}, since $\nwbiratunsmotX$ is defined as a left Bousfield localization of 
	$\unsmot$.
\end{proof}

We will write $\mathrm{Ho}:\unsmot \rightarrow \unstablehomotopyX ,$ for the homotopy functor
associated to the Quillen model category $\unsmot$.

\begin{prop}
		\label{prop.slices=>homotopy}
	Let $n\geq 1$ be an arbitrary integer.  Then:
		\begin{enumerate}
			\item	\label{prop.slices=>homotopy.a}	There exists a unique functor
						\[	t_{n} : \unstablehomotopyX \rightarrow \unstablehomotopyX
						\]
					such that the following diagram commutes:
						\[	\xymatrix{\unsmot \ar[r]^-{\tau_{n}} 
									\ar[d]_-{\mathrm{Ho}} & \unsmot \ar[d]^-{\mathrm{Ho}}\\
									\unstablehomotopyX \ar[r]_-{t_{n}} & 
									\unstablehomotopyX}
						\]
			\item	\label{prop.slices=>homotopy.b}	There exists a unique functor
						\[	s_{n} : \unstablehomotopyX \rightarrow \unstablehomotopyX
						\]
					such that the following diagram commutes:
						\[	\xymatrix{\unsmot \ar[r]^-{\tilde{s}_{n}} 
									\ar[d]_-{\mathrm{Ho}} & \unsmot \ar[d]^-{\mathrm{Ho}}\\
									\unstablehomotopyX \ar[r]_-{s_{n}} & 
									\unstablehomotopyX}
						\]
		\end{enumerate}
\end{prop}
\begin{proof}
	It follows immediately from Corollary \ref{cor.slices.preserve.wequivs} together
	with the universal property of the homotopy functor $\mathrm{Ho}$ for a Quillen
	model category (see \cite[I.1 Defs. 5, 6 and Thm. 1]{MR0223432}).	 
\end{proof}

\begin{defi}
		\label{def.unstable.slice}
	We will call $s_{n}$ the \emph{unstable $n$-slice}, and $s_{0}=t_{1}$ will be
	the \emph{unstable zero slice}.
\end{defi}

\begin{prop}
	\label{slices.fit.fibreseqs}
Let $A\in \unstablehomotopyX$ be an arbitrary simplicial presheaf and $n\geq 1$ an arbitrary integer.
Then the following is a fibre sequence in $\unstablehomotopyX$ (see \cite[I.3 Def. 1]{MR0223432}):
	\[	\xymatrix{s_{n}A \ar[r]^-{i_{n}^{A}}& t_{n+1}A \ar[r]^-{\nu _{n}^{A}}& t_{n}A}
	\]
Furthermore, the map $s_{n}A\rightarrow \ast$ is a fibration in $\nwbiratunsmotX$.
\end{prop}
\begin{proof}
By construction (see Definition \ref{def.constr.slice} and Proposition \ref{prop.properties.slice.tower}(\ref{prop.properties.slice.tower.b}))
the following diagram is cartesian in $\unsmot$:
	\begin{align*}
		\xymatrix{s_{n}A \ar[r]^-{i_{n}^{A}} \ar[d]& t_{n+1}A \ar[d]^-{\nu _{n}^{A}}\\
				 \ast \ar[r] & t_{n}A}
	\end{align*} 
and $\nu _{n}^{A}$ is a fibration in $\nwbiratunsmotX$.  Hence we conclude that $s_{n}A\rightarrow \ast$ is a fibration
in $\nwbiratunsmotX$ and that
$s_{n}A \rightarrow t_{n+1}A \rightarrow t_{n}A$
is a fibre sequence in $\homotnwbiratunsmotX$.  On the other hand, by Theorem 
\ref{tower.Quillenfunctors} we deduce that
$s_{n}A \rightarrow t_{n+1}A \rightarrow t_{n}A$ is also a fibre sequence in $\unstablehomotopyX$.
\end{proof}

\begin{prop}
	\label{tq.respect.order}
Let $A\in \unstablehomotopyX$ be an arbitrary simplicial presheaf and $n\geq j\geq 1$ arbitrary integers.
Then the maps in the following commutative diagram are isomorphisms in $\unstablehomotopyX$:
	\[	\xymatrix{&t_{j}A \ar[dl]_-{t_{j}(\tau_{\leq n+1}^{A})} \ar[dr]^-{t_{j}(\tau_{\leq n}^{A})}&\\
			t_{j}(t_{n+1}A) \ar[rr]^-{t_{j}(\nu _{n}^{A})}&& t_{j}(t_{n}A)}
	\]
\end{prop}
\begin{proof}
By Lemma \ref{lem.slices.preserve.wequivs}\eqref{lem.slices.preserve.wequivs.a} it suffices to show that the maps $\tau_{\leq n+1}^{A}$, $\tau_{\leq n}^{A}$ and $\nu _{n}^{A}$ are weak
equivalences in $\jminonewbiratunsmotX$.  But this follows from Lemma \ref{lemma.weakequivalences.under.tower} together with Proposition 
\ref{prop.properties.slice.tower}\eqref{prop.properties.slice.tower.a} and Corollary \ref{cor.slice.tower.derived.a}, since $n\geq j$.
\end{proof}

\begin{cor}
	\label{cor.tq.respect.order}
Let $A\in \unstablehomotopyX$ be an arbitrary simplicial presheaf and $n\geq j\geq 1$ arbitrary integers.
Then $t_{j}A\cong t_{j}(t_{n}A)$ in $\unstablehomotopyX$.
\end{cor}
\begin{proof}
This follows directly from Proposition \ref{tq.respect.order}.
\end{proof}

The main result of this section is the following:

\begin{thm}
		\label{thm.slice.tower.derived}
	Let $A\in \unsmot$ be an arbitrary simplicial presheaf and $n\geq 1$
	an arbitrary integer.  Then:
		\begin{enumerate}
			\item	\label{thm.slice.tower.derived.a}	The following tower is functorial in $\unsmot$
						\[	\xymatrix@C=1.5pc{&&& \ar[d]|{\tau _{\leq n}^{A}} 
										\ar[dl]|{\tau _{\leq n+1}^{A}} \ar[dlll] A 
										\ar[drr]|{\tau _{\leq 2}^{A}} \ar[drrr]|{\tau _{\leq 1}^{A}} &&&& \\
										 \varprojlim \tau _{n}(A) \ar[r] &
										\cdots \ar[r]_-{\nu _{n+1}^{A}} &  \tau _{n+1}(A) \ar[r]_-{\nu _{n}^{A}} & 
										\tau _{n}(A) \ar[r]_-{\nu _{n-1}^{A}} & 
										\cdots \ar[r]_-{\nu _{2}^{A}} & \tau _{2}(A) \ar[r]_-{\nu _{1}^{A}}& 
										\tau _{1}(A) \ar[r]& \ast
										}
						\]
			\item	\label{thm.slice.tower.derived.b}	The following tower is functorial in $\unstablehomotopyX$
						\[	\xymatrix@C=1.5pc{&&& \ar[d]|{\tau _{\leq n}^{A}} 
										\ar[dl]|{\tau _{\leq n+1}^{A}} \ar[dlll] A 
										\ar[drr]|{\tau _{\leq 2}^{A}} \ar[drrr]|{\tau _{\leq 1}^{A}} &&& \\
										\mathrm{ho}\! \varprojlim t _{n}(A) \ar[r] &
										\cdots \ar[r]_-{\nu _{n+1}^{A}} &  t _{n+1}(A) \ar[r]_-{\nu _{n}^{A}} & 
										t _{n}(A) \ar[r]_-{\nu _{n-1}^{A}} & 
										\cdots \ar[r]_-{\nu _{2}^{A}} & t _{2}(A) \ar[r]_-{\nu _{1}^{A}}& 
										t _{1}(A)	\ar[r] & \ast
										}
						\]
			\item	\label{thm.slice.tower.derived.c} $\tau _{\leq n}^{A}$ is a weak equivalence in $\nminonewbiratunsmotX$
				and $\nu _{n}^{A}$ (resp. $t_{1}(A)\rightarrow \ast $) is a fibration in $\nwbiratunsmotX$
				(resp. $\zerowbiratunsmotX$).
		\end{enumerate}
\end{thm}
\begin{proof}
	(\ref{thm.slice.tower.derived.a}) follows directly from Definition 
	\ref{def.slicetower}, (\ref{thm.slice.tower.derived.b}) follows 
	from Definition \ref{def.slicetower} together with Proposition \ref{prop.slices=>homotopy}, and (\ref{thm.slice.tower.derived.c}) follows from
	Proposition \ref{prop.properties.slice.tower}.
\end{proof}

	We will call the tower in (\ref{thm.slice.tower.derived.b}), the \emph{unstable slice tower}.	

\begin{rmk}
		\label{rmk.gen.smoothness}
	If the base scheme is of the form $X=\mathrm{Spec}\; k$, with $k$ a perfect field, then
	an immediate consequence of Proposition \ref{prop.genericsmoothness-Quillenequiv} is
	the fact that all the constructions in this section are canonically isomorphic when they are
	carried out in the birational motivic categories $\nbiratunsmotX$.
\end{rmk}

\begin{rmk}
		\label{rmk.generic.tower}
	Notice that all the results in this section, and in particular Theorem \ref{thm.slice.tower.derived}, 
	hold if we have the following data:
	\begin{enumerate} 
		\item a model category $\mathcal N$ which is cellular and left proper.
		\item for every $n\geq 0$, a set of maps $L_{n}$ in $\mathcal N$ such that
			every weak equivalence in $L_{n+1}\mathcal N$ is also a weak equivalence in $L_{n}\mathcal M$,
			where $L_{n+1}\mathcal N$ (resp. $L_{n}\mathcal N$) is the left Bousfield localization of 
			$\mathcal N$ with respect to $L_{n+1}$ (resp. $L_{n}$).
	\end{enumerate}
\end{rmk}

\begin{rmk}
	\label{rmk.change.model.str}
In particular, if we choose a different Quillen model structure $\unsmot '$ for the category of pointed simplicial
presheaves on $Sm_{X}$ such that $\unsmot '$ and $\unsmot$ are Quillen equivalent, then all the results in 
this section hold for $\unsmot '$ as well, and the corresponding unstable slice towers of Theorem
\ref{thm.slice.tower.derived}\eqref{thm.slice.tower.derived.b} for $\unsmot '$ and $\unsmot$ are 
canonically isomorphic.
\end{rmk}
		
\end{section}

%% file: sect3_unstsf.tex
\begin{section}{The Unstable Slice Spectral Sequence}
		\label{sect-2}

In this section, we describe the spectral sequence associated to the unstable slice tower
constructed in Theorem \ref{thm.slice.tower.derived}.  We follow the approach
of Bousfield-Kan \cite[chapter IX]{MR0365573}.

In the rest of this section, $A$ (resp. $K$) will be an arbitrary fibrant (resp. cofibrant)
simplicial presheaf in $\unsmot$.  The set of morphisms between two objects in the motivic unstable homotopy
category $\unstablehomotopyX$ will be denoted with square brackets $[-,-]$.

\begin{defi}
		\label{def.E1-term}
	Given arbitrary integers $q\geq p\geq 0$, we consider:
		\begin{enumerate}
			\item \label{def.E1-term.a}
				$ D_{1}^{p,q}(K,A)=[S^{q-p}\wedge K, t_{p+1}A]
				$
			\item \label{def.E1-term.b}
				$	E_{1}^{p,q}(K,A)=[S^{q-p}\wedge K, s_{p}A]
				$
		\end{enumerate}
	and the following maps:
	\begin{enumerate}
			\item \label{thm.extended.uns.slice.exact.couple.c.a}
								$f^{p,q}(K,A):D_{1}^{p+1,q+1}(K,A)
								\rightarrow D_{1}^{p,q}(K,A)$ which is induced by
								$\nu _{p+1}^{A}$ (see Definition \ref{def.slicetower}):
									\[	\xymatrix{[S^{q-p}\wedge K, t_{p+2}A]
									 		\ar[r]^-{(\nu _{p+1}^{A})_{\ast}} & 
											[S^{q-p}\wedge K, t_{p+1}A]}
									\]
			\item \label{thm.extended.uns.slice.exact.couple.c.b} 
				$g^{p,q}(K,A):D_{1}^{p,q}(K,A)
				\rightarrow E_{1}^{p+1,q}(K,A)$ which is only defined for $q-p\geq 1$: 
	\[	\xymatrix{[S^{q-p}\wedge K, t_{p+1}A]\cong
		[S^{q-p-1}\wedge K, \Omega _{S^{1}}t_{p+1}A] \ar[r]& [S^{q-p-1}\wedge K, s_{p+1}A]}
	\]
and is induced by the boundary in the fibre sequence which characterizes the  unstable $p+1$-slice
(see Proposition \ref{slices.fit.fibreseqs}):
	\[ s_{p+1}A\rightarrow t_{p+2}A \rightarrow t_{p+1}A .
	\]
	
			\item \label{thm.extended.uns.slice.exact.couple.c.c}
								$h^{p,q}(K,A):E_{1}^{p,q}(K,A)
								\rightarrow D_{1}^{p,q}(K,A)$ which is induced by
								$i _{p}^{A}$ (see Proposition \ref{prop.slice=>functor}):
									\[	\xymatrix{[S^{q-p}\wedge K, s_{p}A]
									 		\ar[r]^-{(i_{p}^{A})_{\ast}} & 
											[S^{q-p}\wedge K, t_{p+1}A]}
									\]
								for $p\geq 1$, and $h^{0,q}(K,A)=id$ is the identity.
		\end{enumerate}
\end{defi}

\begin{rmk}
		\label{rmk.E1-term}
	Notice that $D_{1}^{p,q}(K,A), E_{1}^{p,q}(K,A)$ are just pointed sets for $p=q$, groups for $q=p+1$ and
	abelian groups for $q\geq p+2$.
\end{rmk}

\begin{thm}
		\label{thm.extended.uns.slice.exact.couple}
	The construction considered in Definition \ref{def.E1-term} forms
	an \emph{extended} exact couple (see \cite[section 4.2, chapter IX]{MR0365573}):
		\begin{equation}	
				\label{eqn.thm.extended.uns.slice.exact.couple}
			\begin{array}{c}
				\xymatrix{ D_{1}(K,A) \ar[rr]^-{f}&& D_{1}(K,A)\ar[dl]^-{g}\\
										& E_{1}(K,A) \ar[ul]_-{h} &}
			\end{array}
		\end{equation}
					
\end{thm}
\begin{proof}
	To construct the exact couple we consider the unstable slice tower for $A$ (see
	Theorem \ref{thm.slice.tower.derived}(\ref{thm.slice.tower.derived.a})):
		\[	\xymatrix@C=1.5pc{&&& \ar[d]|{\tau _{\leq n}^{A}} \ar[dl]|{\tau _{\leq n+1}^{A}} \ar[dlll] A 
										\ar[drr]|{\tau _{\leq 2}^{A}} \ar[drrr]|{\tau _{\leq 1}^{A}} &&&& \\
										 \varprojlim \tau _{n}(A) \ar[r] &
										\cdots \ar[r]_-{\nu _{n+1}^{A}} &  \tau _{n+1}(A)
										 \ar[r]_-{\nu _{n}^{A}} & 
										\tau _{n}(A) \ar[r]_-{\nu _{n-1}^{A}} & 
										\cdots \ar[r]_-{\nu _{2}^{A}} & \tau _{2}(A) \ar[r]_-{\nu _{1}^{A}}& 
										\tau _{1}(A) \ar[r]& \ast
										}
						\]
	Combining Theorems \ref{thm.slice.tower.derived}\eqref{thm.slice.tower.derived.c} and
	\ref{tower.Quillenfunctors}\eqref{tower.Quillenfunctors.b} we deduce that the horizontal arrows in the diagram
	above are fibrations in $\unsmot$.
	Since $\unsmot$ is a simplicial model category (see \cite[Prop. 2.3.7]{MR2807904}) and  $K$ is cofibrant in 
	$\unsmot$, we obtain the following tower of fibrations of pointed simplicial sets
		\[	\xymatrix@C=1pc{&& \ar[d]|{(\tau _{\leq n}^{A})_{\ast}}  \ar[dll] Map_{\ast}(K,A) 
										\ar[drr]|{(\tau _{\leq 1}^{A})_{\ast}}  &&& \\
										 Map_{\ast}(K,\varprojlim \tau _{n}(A)) \ar[r] &
										\cdots 
										 \ar[r]_-{(\nu _{n}^{A})_{\ast}} & 
										Map_{\ast}(K,\tau _{n}(A)) \ar[r]_-{(\nu _{n-1}^{A})_{\ast}} & 
										\cdots 
										 \ar[r]_-{(\nu _{1}^{A})_{\ast}}& 
										Map_{\ast}(K,\tau _{1}(A)) \ar[r]& \ast
										}
						\]
	Finally, we obtain the extended exact couple using the construction of Bousfield-Kan 
	\cite[chapter IX]{MR0365573} together with the following canonical isomorphisms
	(see \cite[Lem. 6.1.2]{MR1650134}) that exist in any simplicial model category with $K$ cofibrant
	and $B$ fibrant
		\[	[S^{r}\wedge K,B]=\Hom _{\unstablehomotopyX}(S^{r}\wedge K,B)
				\cong \pi _{r,\eta} Map_{\ast}(K,B)
		\]
	 where $\eta$ is the base point of
	$Map_{\ast}(K,B)$ and for $r\geq 1$, $\pi _{r,\eta}$ denotes the $r$-th homotopy
	group considering $\eta$ as base point.
\end{proof}

\end{section}

%% file: sect5_unstsf.tex
\begin{section}{Transfers and Kahn-Sujatha Birational Motives}
		\label{sect-4}

	In this section we study the unstable analogue of the Kahn-Sujatha construction of birational 
	motives \cite{K-theory/0596}.	
	
	Let $Cor(X)$ denote the Suslin-Voevodsky category of finite correspondences over $X$; having same
	objects as $Sm_{X}$, morphisms $c(Y,Z)$ given by the group of finite relative cycles on
	$Y\times _{X}Z$ over $Y$ \cite{MR1764199}, and composition as in 
	\cite[p.\,673 diagram (2.1)]{MR2804268}.
	A presheaf with transfers is an additive contravariant functor from $Cor(X)$ to the category of abelian groups.
	Let $\unsmottran$ denote the category of simplicial presheaves with transfers.  
	
	We will write $Y^{tr}$ for the
	simplicial presheaf with transfers represented by $Y$ in $Sm_{X}$.  By taking the graph of a morphism in
	$Sm_{X}$, we obtain a functor $\Gamma :Sm_{X}\rightarrow Cor(X)$.  This induces a forgetful functor:
		\[	\mathcal U : \unsmottran \rightarrow \unsmot
		\]
	which admits a left adjoint
		\[	\mathbb Z _{tr}:\unsmot \rightarrow \unsmottran
		\]
	such that $\mathbb Z_{tr}(Y_{+})=Y^{tr}$ for $Y\in Sm_{X}$.
	
	Moreover, the adjunction \cite[Thm. 8 and Lem. 9]{MR2435654}:
		\begin{align}
				\label{tr.bas.adj}	
			(\mathbb Z _{tr},\mathcal U, \varphi ):\unsmot \rightarrow \unsmottran
		\end{align}
	induces a Quillen model structure on $\unsmottran$ 
	(see Definition \ref{def.ind.Qmod}).  
	We will abuse notation, and write $\unsmottran$ 
	(resp. $\unstablehomotopyX ^{tr}$)
	for the category of simplicial presheaves with transfers equipped with this particular Quillen model structure
	(resp. for the associated homotopy category).
	
	Proceeding as in Definition \ref{def.localizationAmod-Birat} (resp. \ref{def.localizationAmod-weaklyBirat}) we can consider for 
	$n\geq 0$ the left Bousfield localization $\nbiratunsmotX ^{tr}$ (resp. $\nwbiratunsmotX ^{tr}$) of $\unsmottran$ with respect 
	to the set $\mathbb Z _{tr}(B_{n})$ (resp. $\mathbb Z _{tr}(WB_{n})$).  Mimicking the same arguments as in sections 
	\ref{Introd2} and \ref{sect-1}, we deduce that $\nbiratunsmotX ^{tr}$ and $\nwbiratunsmotX ^{tr}$ are Quillen equivalent
	when the base scheme is a perfect field (see Proposition 
	\ref{prop.genericsmoothness-Quillenequiv}), we obtain towers as in 
	Theorems \ref{tower.Quillenfunctors}, \ref{thm.slice.tower.derived} and finally an 
	extended exact couple as in Definition \ref{def.E1-term}  (see Remark \ref{rmk.generic.tower}).
	Thus, we conclude:
	
\begin{prop}
	\label{uns.tower.tr}
Given $A\in \unstablehomotopyX ^{tr}$ we obtain a corresponding unstable slice tower with transfers,
which is functorial in $\unstablehomotopyX ^{tr}$:
	\begin{align*} 
	 \xymatrix@C=1.5pc{&&& \ar[d]|{\rho _{\leq n}^{A}} \ar[dl]|{\rho _{\leq n+1}^{A}} \ar[dlll] A 
		\ar[drr]|{\rho _{\leq 2}^{A}} \ar[drrr]|{\rho _{\leq 1}^{A}} &&& \\
		\mathrm{ho}\! \varprojlim t _{n}^{tr}(A) \ar[r] & \cdots \ar[r]_-{\mu _{n+1}^{A}} &  t _{n+1}^{tr}(A) 
		\ar[r]_-{\mu _{n}^{A}} & t _{n}^{tr}(A) \ar[r]_-{\mu _{n-1}^{A}} & \cdots \ar[r]_-{\mu _{2}^{A}} & t _{2}^{tr}(A) 
		\ar[r]_-{\mu _{1}^{A}} & t _{1}^{tr}(A)	\ar[r] & \ast}
	\end{align*}
where $t _{n}^{tr}(A)$ is fibrant in $\nminonewbiratunsmotX ^{tr}$, $\rho _{\leq n}^{A}$ is a weak equivalence 
in $\nminonewbiratunsmotX ^{tr}$
and $\mu _{n}^{A}$ (resp. $t_{1}^{tr}(A)\rightarrow \ast$) is a fibration in $\nwbiratunsmotX ^{tr}$ 
(resp. $\zerowbiratunsmotX ^{tr}$).
\end{prop}
	
	 By construction, for every $n\geq 0$ the Quillen adjunction:
			\[	(\mathbb Z _{tr},\mathcal U, \varphi ):\unsmot \rightarrow \unsmottran
			\]
	induces a pair of Quillen adjunctions between the corresponding birational motivic homotopy 
	categories:
		\begin{align}	
				\label{tr.adj}	
			\begin{aligned}
				(\mathbb Z _{tr},\mathcal U, \varphi ): & \nbiratunsmotX \rightarrow 
				\nbiratunsmotX ^{tr}\\
				(\mathbb Z _{tr},\mathcal U, \varphi ): & \nwbiratunsmotX \rightarrow \nwbiratunsmotX ^{tr} 
			\end{aligned}
		\end{align}
	and passing to the associated homotopy categories, we obtain the corresponding derived adjunctions:
		\begin{align}	
				\label{tr.adj.der}
			\begin{aligned}
				(\mathbb Z _{tr},b_{tr}^{(n)}, \varphi ): & \homotnbiratunsmotX \rightarrow 
				\homotnbiratunsmotX ^{tr}\\
				(\mathbb Z _{tr},wb_{tr}^{(n)}, \varphi ): & \homotnwbiratunsmotX \rightarrow 
				\homotnwbiratunsmotX ^{tr} 
			\end{aligned}
		\end{align}
	where $b_{tr}^{(n)}$ (resp. $wb_{tr}^{(n)}$) is a fibrant replacement functor in $\nbiratunsmotX ^{tr}$
	(resp. $\nwbiratunsmotX ^{tr}$).
	
\begin{rmk}
		\label{rmk.KS}
	 By the work of R\"ondigs and {\O}stv{\ae}r \cite{MR2435654}, 
	$\homotzerobiratunsmotX ^{tr}$ is the natural unstable analogue of the Kahn-Sujatha 
	birational category of motives \cite{K-theory/0596}.
\end{rmk}
	
	Since the maps in $B_{n}$ and $WB_{n}$ (see Definitions 
	\ref{def.localizing-maps} and \ref{def.localizing-maps2})
	are cofibrations in $\unsmot$ and the model structure in $\unsmottran$ is induced by the adjunction \eqref{tr.bas.adj}
	(see Definition \ref{def.ind.Qmod}),
	we observe that the Quillen model structures $\nbiratunsmotX ^{tr}$, $\nwbiratunsmotX
	^{tr}$ are induced by the adjunctions \eqref{tr.adj} (see \cite[Prop. 3.1.12 and Thm. 11.3.2]{MR1944041} and
	Definition \ref{def.ind.Qmod}).  Thus, we conclude:
	
\begin{prop}
		\label{prop.loc.resp.trans}
	Let $f$ be a map in $\nbiratunsmotX ^{tr}$ (resp. $\nwbiratunsmotX ^{tr}$).  Then:
	\begin{enumerate}
	 	\item $f$ is a weak equivalence in $\nbiratunsmotX ^{tr}$ (resp. $\nwbiratunsmotX ^{tr}$) if and only if $\mathcal U f$ is a 
				weak equivalence in $\nbiratunsmotX$ (resp. $\nwbiratunsmotX$).
		\item	$f$ is a fibration in $\nbiratunsmotX ^{tr}$ (resp. $\nwbiratunsmotX ^{tr}$) if and only if $\mathcal U f$ is
				 a fibration in $\nbiratunsmotX$ (resp. $\nwbiratunsmotX$).
	\end{enumerate}
\end{prop}

	Now we can state the main result of this section.
	
\begin{thm}
		\label{thm.comp.tran}
	Let $A$ in $\unstablehomotopyX ^{tr}$ be an arbitrary simplicial presheaf with transfers.  Then:
		\begin{enumerate}
	\item \label{thm.comp.tran.a}  The forgetful functor $\mathcal U:\unsmottran \rightarrow \unsmot$ induces an 
		isomorphism in $\unstablehomotopyX$ between the unstable slice tower with transfers for $A$ 
		(see Proposition \ref{uns.tower.tr}) and 
	the corresponding unstable slice tower for $\mathcal U (A)$ (see Theorem 
	\ref{thm.slice.tower.derived}\eqref{thm.slice.tower.derived.b}).
	\item \label{thm.comp.tran.b} For $n\geq 0$, the unstable slices $s_{n}(\mathcal U (A))$ of $\mathcal U (A)$
	admit a canonical structure of simplicial presheaves with transfers.
		\end{enumerate}
\end{thm}
\begin{proof}
\eqref{thm.comp.tran.a}:  By the functoriality of the unstable slice tower (see Theorem \ref{thm.slice.tower.derived})
and Lemma \ref{lem.slices.preserve.wequivs}, it suffices to show that for $n\geq 0$ the maps 
	\[ \rho_{\leq n}^{A}:A\rightarrow t_{n}^{tr}(A)
	\]
are weak equivalences in $\nminonewbiratunsmotX$, i.e. after forgetting transfers.  But this follows
from Propositions \ref{uns.tower.tr} and \ref{prop.loc.resp.trans}.
	
	\eqref{thm.comp.tran.b}:  The case of the zero slice $s_{0}=t_{1}$ follows from \eqref{thm.comp.tran.a} above.  If
	$n\geq 1$, let $s_{n}^{tr}(A)$ denote the unstable slice with transfers of $A$.  
	Thus, by Proposition \ref{slices.fit.fibreseqs}
	the following is a fibre sequence in $\unstablehomotopyX ^{tr}$:
		\[	\xymatrix{s_{n}^{tr}(A)  \ar[r] & t_{n+1}^{tr}(A) \ar[r]^-{\mu _{n}^{A}} & t_{n}^{tr}(A)}
		\]
	and by \eqref{thm.comp.tran.a} above there is a commutative diagram where the rows are fibre sequences in 
	$\unstablehomotopyX$ such that $t_{n}(\mathcal U (\rho _{\leq n}^{A}))$, 
	$t_{n+1}(\mathcal U (\rho _{\leq n+1}^{A}))$ are isomorphisms in $\unstablehomotopyX$.
	\[	\xymatrix{s_{n}(\mathcal U (A))  \ar[r] \ar[d]_-{\alpha}& 
			t_{n+1}(\mathcal U (A)) \ar[r]^-{\nu _{n}^{\mathcal U (A)}} \ar[d]_-{t_{n+1}(\mathcal U 
			(\rho _{\leq n+1}^{A}))}& t_{n}(\mathcal U (A)) \ar[d]^-{t_{n}(\mathcal U (\rho _{\leq n}^{A}))}\\
			\mathcal U (s_{n}^{tr}(A))  \ar[r] & \mathcal U (t_{n+1}^{tr}(A)) \ar[r]_-{\mathcal U (\mu _{n}^{A})} & 
			\mathcal U (t_{n}^{tr}(A))}
	\]
	Hence, by \cite[I.3 Prop. 5(iii)]{MR0223432} we conclude that $\alpha$ is also
	 an isomorphism in $\unstablehomotopyX$.
\end{proof}

\begin{rmk}
		\label{rmk.no.gen.trans}
	The unstable slices $s_{n}A$ of an arbitrary simplicial presheaf $A\in \unstablehomotopyX$ do not 
	admit transfers in general.  The example of Levine \cite[\S 2]{MR2804260} works as well in the unstable 
	setting and shows that if our base
	scheme is a field $k$ and $C$ a smooth projective curve of genus $g>0$ with no rational $k$-points
	then $s_{0}(\mathbb Z _{C})$ does not admit transfers, 
	where $\mathbb Z _{C}$ is the presheaf of abelian groups represented by $C$.
\end{rmk}

\begin{defi}
	\label{def.ind.Qmod}
Let $(F,G,\varphi):\mathcal A \rightarrow \mathcal B$ be a Quillen adjunction between two Quillen model categories
$\mathcal A$, $\mathcal B$.  We say that the Quillen model structure on $\mathcal B$ is induced by the adjunction
$(F,G,\varphi)$ if a map $f$ in $\mathcal B$ is a weak equivalence (resp. fibration) if and only if $G(f)$ is
a weak equivalence (resp. fibration) in $\mathcal A$.
\end{defi}
	
\end{section}

%% file: sect6_unstsf.tex
\begin{section}{Comparison with the Classical Postnikov Tower}
		\label{sect-5}
	
	In this section the base scheme $X$ will be of the form $\mathrm{Spec}(\mathbb C )$,
	where $\mathbb C$ denotes the complex numbers, and $Sm_{X}$ will consist of only
	smooth quasi-projective varieties.  We will write $\unsmot '$ for the category of pointed simplicial
	presheaves on $Sm_{X}$ equipped with the Quillen model structure defined in \cite[Thm.\,A.17]{MR2597741},
	which is Quillen equivalent \cite[Rmk.\,A.21]{MR2597741} to $\unsmot$.  
	By Remarks \ref{rmk.generic.tower} and \ref{rmk.change.model.str}, all
	the results in section \ref{sect-1} hold in $\unsmot '$
	and the constructions are canonically isomorphic after we pass to the respective homotopy
	categories.  We will write $\unstablehomotopyX '$ for the homotopy category associated to
	$\unsmot '$.
	
	Our goal is to 
	study the behavior of
	the unstable slice filtration (see Theorem
	\ref{thm.slice.tower.derived}\eqref{thm.slice.tower.derived.b}) in the motivic unstable homotopy category 
	$\unstablehomotopyX '$ 
	with respect to the functor of $\mathbb C$-points in $Sm_{X}$ and 
	compare the unstable slice 
	filtration with the classical Postnikov tower of algebraic topology.
	
	We will write $\topspc$ for the category of pointed compactly generated topological spaces \cite{MR0210075}
	equipped with the model structure constructed by Quillen (see \cite[Ch.\:2,\:\S 3,\:Thm.\:1]{MR0223432} or 
	\cite[Thm.\:2.4.23]{MR1650134}), and
	$\htop$ for its associated homotopy category.  
	Notice that $\topspc$ is a simplicial model category \cite[Ch.\:2,\:\S 2]{MR0223432}, 
	we will write $Map_{\mathsf{top}}(-,-)$ for the corresponding simplicial mapping space \cite[1.1.6]{MR1944041}.
	
	
For $n\geq 1$, we will write $D^{n}$ for the unit disc contained in $\mathbb R ^{n}$ and
$d_{n}\in \topspc$ will be the standard inclusion $\partial D^{n+1}_{+}\rightarrow D^{n+1}_{+}$
of the $n$-sphere into the $n+1$-unit disc (both with a disjoint base point).  Consider:
	\[L_{n-1}\mathbf T=\{ d_{n}\}.\] 
 We will write $L_{n-1}\topspc$ for the left Bousfield localization of $\topspc$
with respect to $L_{n-1}\mathbf{T}$, which exists by \cite[12.1.4]{MR1944041} and Theorem \ref{Hirsch-Bousloc}.	

Then, mimicking the same
 arguments as for the unstable slice filtration (see Remark \ref{rmk.generic.tower}), we deduce:
	
\begin{prop}
		\label{prop.Postnikov.tower}
	Let $n\geq 1$ be and arbitrary integer.  Then, there exist:
		\begin{enumerate}
	\item \label{prop.Postnikov.tower.a} Functors $P_{n}:\htop \rightarrow \htop$.
	\item \label{prop.Postnikov.tower.b} Natural transformations:
		\[ p_{\leq n}:id\rightarrow P_{n}, \: \:
			p_{n}:P_{n+1}\rightarrow P_{n}
		\]
	such that for every topological space $K\in \topspc$
	the following tower is functorial in $\htop$:
		\[	
		\xymatrix@C=1.5pc{&&& \ar[d]|{p_{\leq n}^{K}} 
		\ar[dl]|{p_{\leq n+1}^{K}} \ar[dlll] K 
		\ar[drr]|{p_{\leq 2}^{K}} \ar[drrr]|{p_{\leq 1}^{K}} &&& \\
	\mathrm{ho}\! \varprojlim P _{n}(K) \ar[r] &	\cdots \ar[r]_-{p _{n+1}^{K}} &  P _{n+1}(K) \ar[r]_-{p_{n}^{K}} &
 	P _{n}(K) \ar[r]_-{p_{n-1}^{K}} & \cdots \ar[r]_-{p_{2}^{K}} & P_{2}(K) \ar[r]_-{p _{1}^{K}}& P_{1}(K)	\ar[r] & \ast}
		\]
	where $P _{n}(K)$ is fibrant in $L_{n-1}\topspc$, $p_{\leq n}^{K}$ is a weak equivalence in $L_{n-1}\topspc$
	and $p_{n}^{K}$ (resp. $P_{1}(K)\rightarrow \ast$) is a fibration in $L_{n}\topspc$ (resp. $L_{0}\topspc$).
		\end{enumerate}
\end{prop}

The work of Hirschhorn \cite{MR1944041} implies that the tower constructed in Proposition \ref{prop.Postnikov.tower}
coincides with the Postnikov tower in $\topspc$:

\begin{thm}[Hirschhorn]
		\label{prop.Postnikov.agree}
	The tower constructed in Proposition \ref{prop.Postnikov.tower} is the Postnikov tower of $K\in \topspc$.
\end{thm}
\begin{proof}
	Hirschhorn constructs the Postnikov tower in the context of unpointed topological spaces
	using Bousfield localizations with respect to $\partial D^{n+1}\rightarrow D^{n+1}$
	(see 1.5.1, 1.5.2, 1.5.3 and 1.5.4 in \cite{MR1944041}).  Since the maps $\partial D^{n+1}\rightarrow D^{n+1}$
	are cofibrations and the adjunction between unpointed and pointed topological spaces is enriched
	in simplicial sets, we deduce that a pointed space is local with respect to 
	$\partial D^{n+1}_{+}\rightarrow D^{n+1}_{+}$ if and only if it is local (as an unpointed space) with
	respect to $\partial D^{n+1}\rightarrow D^{n+1}$ (see \cite[Prop.\,3.1.12 and Thm.\,11.3.2]{MR1944041}).
	This finishes the proof.
\end{proof}

	If $Y\in Sm_{X}$, we will write $Y(\mathbb C)$ for the set of complex points of $Y$, which
	is a complex analytic manifold.
	Hence, in particular $Y(\mathbb C )_{+}\in \topspc$.  
	
\begin{prop}(\cite{MR2597741})
	 \label{prop.cpxreal.lQuillen}
	The functor $\mathbb C _{R}$ of complex points is the left adjoint in a Quillen adjunction:
		\begin{align*}
				(\mathbb C _{R}, Sing,\varphi):\unsmot ' \rightarrow \topspc
		\end{align*}
	where $Sing(K)$ is the simplicial presheaf
		\[	Y\in Sm_{X}\mapsto Map_{\mathsf{top}}(Y(\mathbb C)_{+}, K).
		\]
\end{prop}
\begin{proof}
This follows from \cite[Thm A.23]{MR2597741}.
\end{proof}

Now we are ready to establish a connection between our unstable slice filtration and the classical Postnikov tower in algebraic topology.  Let $B_{n}(\mathbb C)$ be the following set of maps in $\topspc$:
	\begin{align}
		\label{diag.cplx.openimm}
			B_{n}(\mathbb C)=
			\{ \mathbb C _{R}(\iota _{U,Y}): U(\mathbb C)_{+}\rightarrow Y(\mathbb C)_{+} \mid
			\iota _{U,Y}\in B_{n} \}
	\end{align}
where we are considering $B_{n}$ as a set of maps in $\unsmot '$ (see Definition \ref{def.localizing-maps}).  
We define $B_{n}\topspc$ to be
the left Bousfield localization of $\topspc$ with respect to the set $B_{n}(\mathbb C)$, which exists by
\cite[12.1.4]{MR1944041} and Theorem \ref{Hirsch-Bousloc}.
We will write $B_{n}\htop$ for its associated homotopy category. 

Then, mimicking the same
 arguments as for the unstable slice filtration (see Remark \ref{rmk.generic.tower}), we deduce:
		
\begin{prop}	
	\label{prop.complexPtower}
Let $n\geq 1$ be and arbitrary integer.  Then, there exist:
\begin{enumerate}
	\item \label{prop.complexPtower.a} Functors $ t_{n}^{\mathbb C}:\htop \rightarrow \htop$.
	\item \label{prop.complexPtower.b} Natural transformations:
				\[ \mathrm{r}_{\leq n}:id\rightarrow t_{n}^{\mathbb C}, \: \: \mathrm{r}_{n}:t_{n+1}^{\mathbb C}\rightarrow t_{n}^{\mathbb C}
		\]
	such that for every topological space $K\in \topspc$
		the following tower is functorial in $\htop$:
		\[	\xymatrix@C=1.5pc{&&& \ar[d]|{\mathrm{r}_{\leq n}^{K}} 
		\ar[dl]|{\mathrm{r}_{\leq n+1}^{K}} \ar[dlll] K 
		\ar[drr]|{\mathrm{r}_{\leq 2}^{K}} \ar[drrr]|{\mathrm{r}_{\leq 1}^{K}} &&& \\
		\mathrm{ho}\! \varprojlim t_{n}^{\mathbb C}(K) \ar[r] &	\cdots \ar[r]_-{\mathrm{r}_{n+1}^{K}} &  t_{n+1}^{\mathbb C}(K) 
		\ar[r]_-{\mathrm{r}_{n}^{K}} & t_{n}^{\mathbb C}(K) \ar[r]_-{\mathrm{r}_{n-1}^{K}} & \cdots \ar[r]_-{\mathrm{r}_{2}^{K}} &
		 t_{2}^{\mathbb C}(K) \ar[r]_-{\mathrm{r}_{1}^{K}}& t_{1}^{\mathbb C}(K)	\ar[r] & \ast}
		\]
where $t_{n}^{\mathbb C}(K)$ is fibrant in $B_{n-1}\topspc$, $\mathrm{r}_{\leq n}^{K}$ is a weak equivalence in 
$B_{n-1}\topspc$ and $\mathrm{r}_{n}^{K}$ (resp. $t_{1}^{\mathbb C}(K)\rightarrow \ast$) is a fibration in 
$B_{n}\topspc$ (resp. $B_{0}\topspc$).
\end{enumerate}
\end{prop}

\begin{rmk}
		\label{rmk.real.WBreal}
By Proposition \ref{prop.genericsmoothness-Quillenequiv}, the tower in 
Proposition \ref{prop.complexPtower} is identical if we replace $B_{n}$ with 
$WB_{n}$ (see Definition \ref{def.localizing-maps2}).
\end{rmk}

Notice that by construction, for every $n\geq 0$ the Quillen adjunction 
(see Proposition \ref{prop.cpxreal.lQuillen}):
				\[ (\mathbb C _{R}, Sing,\varphi):\unsmot ' \rightarrow \topspc
				\]
induces a Quillen adjunction between the corresponding birational  categories:
		\begin{align}	
				\label{cpreal.adj}	
				(\mathbb C _{R}, Sing,\varphi):  \nbiratunsmotX ' \rightarrow B_{n}\topspc
		\end{align}
and passing to the associated homotopy categories, we obtain the corresponding derived adjunction:
		\begin{align}	
				\label{cpreal.adj.der}
				(\mathbb C _{R},b_{top}^{(n)}, \varphi ): & \mathcal H '(B_{n}) \rightarrow B_{n}\htop
		\end{align}
where $b_{top}^{(n)}$  is a fibrant replacement functor in $B_{n}\topspc$.

The main result of this section is the following:

\begin{thm}
		\label{thm.Post.real.comp}
	For every $n\geq 0$, the identity functor:
	\[ id: L_{2n}\topspc \rightarrow B_{n}\topspc\]
	is a left Quillen functor and a Quillen equivalence.
\end{thm}
\begin{proof}
	We observe that there is a commutative diagram in $\topspc$:
		\[	\xymatrix{\partial D^{2n+2} \ar[d] \ar[r]^-{d_{2n+1}}& D^{2n+2} \ar[d]\\
						\mathbb C^{n+1}\backslash 0 \ar[r]^-{\tilde{d}_{n}}& \mathbb C^{n+1}}
		\]
	where the vertical arrows are weak equivalences in $\topspc$.  Thus, by construction the left Bousfield
	localization of $\topspc$ with respect to $d_{2n+1}$ and $\tilde{d}_{n}$ are identical \cite[4.1.1 and 4.1.2]{MR1944041}.  On the other hand,
	$\iota _{n}: \mathbb A^{n+1}_{\mathbb C}\backslash 0\rightarrow \mathbb A^{n+1}_{\mathbb C}$ is in 
	$B_{n}$, and $\mathbb C _{R}(\iota _{n})=\tilde{d}_{n}$.  By the universal property of left Bousfield 
	localizations (see \cite[3.3.19.(1) and 3.1.1.(1)]{MR1944041}), it follows that the identity functor is a left
	Quillen functor.
	
	Using the universal property of left Bousfield localizations again, we observe that in order to conclude that we have a Quillen equivalence
	it is enough to show that every map in $B_{n}(\mathbb C)$ (see \eqref{diag.cplx.openimm}) 
	is a weak equivalence in $L_{2n}\topspc$.   On the other hand, by Proposition \ref{prop.genericsmoothness-Quillenequiv}
	and Remark \ref{rmk.real.WBreal} it suffices to show that
	every map in $WB_{n}(\mathbb C)$  
	is a weak equivalence in $L_{2n}\topspc$.
	But this follows directly from Lemma \ref{lem.qproj.conn}.
\end{proof}

\begin{cor}
		\label{cor.loc.resp.singfun}
Let $K$ in $\topspc$ be a topological space and $n\geq 0$ an arbitrary integer.  Then
the following are equivalent:
\begin{enumerate}
	\item \label{cor.loc.resp.singfun.a} $K$ is fibrant in $L_{2n}\topspc$,
	\item \label{cor.loc.resp.singfun.b} $K$ is fibrant in $B_{n}\topspc$,
\end{enumerate}
\end{cor}
\begin{proof}
	This follows immediately from Theorem \ref{thm.Post.real.comp}.
\end{proof}

\begin{rmk}
	\label{rmk.evenPosttow.alggeom}
Combining Theorems \ref{thm.Post.real.comp} and \ref{prop.Postnikov.agree} we observe that
we have obtained an algebro-geometric construction of the odd
part of the Postnikov tower (see Proposition \ref{prop.Postnikov.tower}):
	\[	
		\xymatrix@C=1.5pc{&&& \ar[d]|{p_{\leq 2n+1}^{K}} 
		\ar[dl]|{p_{\leq 2n-1}^{K}} \ar[dlll] K 
		\ar[drr]|{p_{\leq 3}^{K}} \ar[drrr]|{p_{\leq 1}^{K}} &&& \\
	\mathrm{ho}\! \varprojlim P _{2n+1}(K) \ar[r] &	\cdots \ar[r] &  P _{2n+1}(K) \ar[r] &
 	P _{2n-1}(K) \ar[r] & \cdots \ar[r] & P_{3}(K) \ar[r] & P_{1}(K)	\ar[r] & \ast}
		\]
\end{rmk}



	In the rest of this section we will write $Y$ for a quasi-projective variety in $Sm_{X}$
	with a fixed embedding into projective space, i.e.
	an open immersion $Y\rightarrow \bar{Y}$ and a closed 
	embedding $\bar{Y}\rightarrow \mathbb P ^{N}$ for a suitable $N$.
	
	Given a map $\iota _{U,Y}$ in $B_{n}$ and a hyperplane $H$ in $\mathbb P ^{N}$,
	we will write $e_{1}:H\cap Y\rightarrow Y$, $e_{2}:H\cap U\rightarrow U$ for the closed embedding of the
	corresponding hyperplane sections.

\begin{lem}
	\label{lem.qproj.conn2}
Let $Y\in Sm_{X}$ be a smooth irreducible quasi-projective variety of dimension $d\geq 2$, and
 $\iota _{U,Y}:U_{+}\rightarrow Y_{+}$ a map in $B_{n}$.  Then there exists a hyperplane $H$ in $\mathbb P^{N}$ such that in the following commutative diagram:
		\[	\xymatrix{H\cap U(\mathbb C) _{+} \ar[rr]^-{\mathbb C _{R}(e_{2})} \ar[d]_-{\mathbb C_{R}(h)} &&
					U(\mathbb C)_{+} \ar[d]^-{\mathbb C _{R}(\iota _{U,Y})}\\
			H\cap Y(\mathbb C) _{+}\ar[rr]_-{\mathbb C _{R}(e_{1})}&& Y(\mathbb C)_{+} }
		\]
\begin{enumerate}
	\item \label{lem.qproj.conn2.a} $\mathbb C _{R}(e_{1})$ and $\mathbb C _{R}(e_{2})$ 
	are weak equivalences in $L_{d-2}\topspc$,
	\item \label{lem.qproj.conn2.bb} $H\cap Y$ is smooth of dimension $d-1$, 
	\item \label{lem.qproj.conn2.b} If $n\leq d-2$ then $h:H\cap U _{+}\rightarrow H\cap Y_{+}$ is in $B_{n}$, and 
	\item \label{lem.qproj.conn2.c} If
			$n=d-1$ then $h=id$, i.e. $H\cap (Y\backslash U)=\emptyset$.
\end{enumerate}
\end{lem}
\begin{proof}
	By Proposition 1.5.2 in \cite{MR1944041} in order to prove \eqref{lem.qproj.conn2.a}
	it suffices to show that $\mathbb C _{R}(e_{1})$ (resp. $\mathbb C _{R}(e_{2})$) induces a bijection on 
	the homotopy groups $\pi _{i}$ for
	$i\leq d-2$ and for every choice of base point in $H\cap Y(\mathbb C) _{+}$ (resp. $H\cap U(\mathbb C) _{+}$).
	
	Let $Z=Y\backslash U$ be the closed complement of $U$ in $Y$, with codimension $c$ in $Y$.  Since $\iota _{U,Y}\in B_{n}$
	we observe that $d\geq c\geq n+1$.  Since $Y$ is smooth, by the theorem of Bertini
	\cite[Cor. 3.6]{MR0265371}
	there exists a Zariski open dense subset $\Omega$ of the projective variety $\check{\mathbb P} ^{N}$ of hyperplanes in $\mathbb P ^{N}$
	such that for every
	hyperplane $H$ in $\Omega$, \eqref{lem.qproj.conn2.bb}
	holds and the codimension of $H\cap Z$ in $Y$ is $c+1$.  Thus, we deduce that \eqref{lem.qproj.conn2.b} and \eqref{lem.qproj.conn2.c}
	also hold for any $H$ in $\Omega$.
	
	Moreover, $U$ and $Y$ are smooth, so by the Zariski-Lefschetz theorem of Hamm-L\^{e} \cite[Thm. 1.1.3]{MR820315} we deduce that
	there exists a hyperplane $H$ in $\Omega$ such that
	$Y(\mathbb C)$ (resp. $U(\mathbb C)$)
	has the homotopy type
	of a space obtained from $H\cap Y(\mathbb C)$ (resp. $H\cap U(\mathbb C)$)
	by attaching cells $\partial D^{r}\rightarrow D^{r}$
	with $r\geq d$.  	
	Since all the spaces in the diagram above are $CW$-complexes, we deduce by
	the cellular approximation theorem that $\mathbb C _{R}(e_{1})$ and $\mathbb C _{R}(e_{2})$ 
	induce a bijection on $\pi _{i}$
	for every choice of base point and $i\leq r-2$.
	Hence, we conclude that \eqref{lem.qproj.conn2.a} holds since $d-2\leq r-2$. 
\end{proof}

\begin{lem}
		\label{lem.qproj.stepc}
	Let $Y\in Sm_{X}$ be a smooth irreducible quasi-projective variety of dimension $d$, and
	$\iota _{U,Y}:U_{+}\rightarrow Y_{+}$
	a map in $B_{n}$ with $n>0$.  Then:
		\[	\pi_{1}(\mathbb C _{R}(\iota _{U,Y})):\pi_{1}(U(\mathbb C)_{+})\rightarrow \pi_{1}(Y(\mathbb C)_{+})
		\]
	is an isomorphism for every choice of base point in $U(\mathbb C)_{+}$.
\end{lem}
\begin{proof}
Since the map $\iota _{U,Y}$ preserves the disjoint point $+$,
we can assume that the base point is in $U(\mathbb C)$.  Moreover, $Y$ is irreducible and
$U$ is open in $Y$, so we deduce that $U$ is also irreducible and in particular $U(\mathbb C)$
is path connected.  Hence the result is independent of the choice of base point, and we will omit
any reference to it in the rest of the proof.
We proceed by induction on the dimension $d$ of $Y$.
We observe that $d\geq 2$, since $n>0$.

Let $Z=Y\backslash U$ be the closed complement  of $U$ in $Y$, which has codimension $c\geq n+1$ in $Y$ since
$\iota _{U,Y}\in B_{n}$.  So,
if $d=2$ then $\dim Z=0$, since $n>0$. 
Thus, $Z$ is a disjoint union of closed points (since we are assuming that $Z$ is reduced). We observe that it suffices to consider the case when
$Z$ consists of only one point, since $\iota _{U,Y}$ can be written as a composition of open immersions
where the closed complement consists of only one point.  If $Z=\mathrm{Spec}\; \mathbb C$, then
there exists a pushout square of open immersions in $\topspc$ of the form:
	\begin{align*}
		\xymatrix{D^{4}\backslash 0 \ar[rr] \ar[d]&& D^{4} \ar[d]\\
		U(\mathbb C) \ar[rr]_-{\mathbb C _{R}(\iota _{U,Y})}&& Y(\mathbb C).}
	\end{align*} 
By the theorem of Van Kampen we deduce that $\pi_{1}(\mathbb C _{R}(\iota _{U,Y}))$ is equal to:
	\[	\pi_{1}(U(\mathbb C))\rightarrow \mathrm{colim} \left ( 
			\begin{array}{c} 
				\xymatrix{\pi_{1}(D^{4}\backslash 0)=\ast \ar[rr] \ar[d]&& \pi _{1}(D^{4})=\ast \\
		\pi _{1}(U(\mathbb C)) && }
			\end{array}\right ) \cong \pi_{1}(U(\mathbb C))
	\]
which is an isomorphism, so the claim holds in this case.

If $d\geq 3$, then by Lemma \ref{lem.qproj.conn2} there exists a hyperplane $H$ in $\mathbb{P} ^{N}$ such that
in the following commutative diagram:
	\[	\xymatrix{H\cap U(\mathbb C) _{+} \ar[rr]^-{\mathbb C _{R}(e_{2})} \ar[d]_-{\mathbb C_{R}(h)} &&
					U(\mathbb C)_{+} \ar[d]^-{\mathbb C _{R}(\iota _{U,Y})}\\
			H\cap Y(\mathbb C) _{+}\ar[rr]_-{\mathbb C _{R}(e_{1})}&& Y(\mathbb C)_{+} }
		\]
$\mathbb C _{R}(e_{1})$ and $\mathbb C _{R}(e_{2})$ are weak equivalences in 
$L_{d-2}\topspc$.  We observe that $d-2\geq 3-2=1$,
so Proposition 1.5.2 in \cite{MR1944041} impies that $\mathbb C _{R}(e_{1})$ and $\mathbb C _{R}(e_{2})$
induce a bijection on  the corresponding fundamental groups $\pi _{1}$.

If $n=d-1$, then by Lemma \ref{lem.qproj.conn2}\eqref{lem.qproj.conn2.c}, $\mathbb C_{R}(h)$ is the identity
map; so we conclude that $\mathbb C _{R}(\iota _{U,Y})$ induces an isomorphism on $\pi_{1}$.
On the other hand, if $n\leq d-2$ then by \eqref{lem.qproj.conn2.bb}-\eqref{lem.qproj.conn2.b} 
in Lemma \ref{lem.qproj.conn2} 
we deduce that $h\in B_{n}$ and $H\cap Y$ is smooth of dimension $d-1$.
Thus, by induction on the dimension we conclude that $\mathbb C_{R}(h)$
induces an isomorphism on $\pi _{1}$.  Hence, it follows that $\mathbb C _{R}(\iota _{U,Y})$
also induces an isomorphism on $\pi _{1}$, as we wanted.
\end{proof}

Recall that a topological space $A\in \topspc$ is simply connected if it is path connected
(i.e. $\pi _{0}A=\ast$) and $\pi _{1}(A)=\ast$ for every choice of base point in $A$.

\begin{lem}
		\label{lem.realmflds.stepd}
Let $W'\rightarrow W$ be a closed embedding of real codimension $r\geq 2$ between real smooth manifolds.
Assume that $W$ and $U=W\backslash W'$ are simply connected.  Then the open immersion:
 \begin{align*}
 	\iota :U_{+}\rightarrow W_{+}
 \end{align*}
is a weak equivalence in $L_{r-2}\topspc$.
\end{lem}
\begin{proof}
By Proposition 1.5.2 in \cite{MR1944041} it suffices to show that $\iota$ induces a bijection on 
the homotopy groups $\pi _{i}$ for $i\leq r-2$ and for every choice of base point in $U$.
Since $U$ and $W$ are simply connected, by Whitehead's theorem it suffices to show that the induced map $H_{i}(\iota)$
on singular homology with integral coefficients is an isomorphism for $0\leq i\leq r-2$ and
an epimorphism for $i=r-1$ (see \cite[Thm. 4]{MR0030759}, \cite[p. 181 Thm. 7.13]{MR516508}).
But this is equivalent to:
	\begin{align*}
		H_{i}(W/U)=\ast \quad  \text{ for every } i\leq r-1,
	\end{align*}
where we write $W/U$ for the homotopy cofibre of $\iota$.
Now, by excision and the tubular neighborhood theorem (cf.\,\cite[p.\,117 Cor.\,11.2]{MR0440554}): 
	\begin{align*}
		H_{i}(W/U)\cong H_{i}(Th(N_{W',W}))
	\end{align*}
where $Th(N_{W',W}$) is the Thom space of the normal bundle of the closed embedding $W'\rightarrow W$,
(see Definition \ref{def.Thom.space}).  Since
singular homology satisfies the Mayer-Vietoris property, it suffices to consider the case when $N_{W',W}$ is trivial.
In this case, $Th(N_{W',W})$ is homotopy equivalent to $S^{r}\wedge W'_{+}$, since $N_{W',W}$ is a vector bundle of rank $r$
over $W'$.  Thus, we conclude that for $i\leq r-1$:
	\begin{align*}
		H_{i}(Th(N_{W',W}))\cong H_{i}(S^{r}\wedge W'_{+})=\ast
	\end{align*}
as we wanted.
\end{proof}

\begin{cor}
		\label{cor.qproj.stepd}
Let $V'\rightarrow V$ be a closed embedding of complex codimension $c\geq 1$ between analytic manifolds.
Assume that $V$ and $U=V\backslash V'$ are simply connected.  Then the open immersion:
 \begin{align*}
 	\iota :U_{+}\rightarrow V_{+}
 \end{align*}
is a weak equivalence in $L_{2c-2}\topspc$.
\end{cor}
\begin{proof}
This follows from Lemma \ref{lem.realmflds.stepd}, taking $r=2c$.
\end{proof}

\begin{lem}
		\label{lem.qproj.conn}
	Let $Y\in Sm_{X}$ be a smooth irreducible quasi-projective variety of dimension $d$, and $\iota _{U,Y}:U_{+}\rightarrow Y_{+}$
	a map in $WB_{n}$.  Then:
		\[	\mathbb C _{R}(\iota _{U,Y}):U(\mathbb C)_{+}\rightarrow Y(\mathbb C)_{+}
		\]
	is a weak equivalence in $L_{2n}\topspc$.
\end{lem}
\begin{proof}
By Proposition 1.5.2 in \cite{MR1944041} it suffices to show that 
$\mathbb C _{R}(\iota _{U,Y})$ induces a bijection on the homotopy groups $\pi _{i}$ for
$i\leq 2n$ and for every choice of base point in $U(\mathbb C)_{+}$.  Since the map $\iota _{U,Y}$ preserves the disjoint point $+$,
we can assume that the base point is in $U(\mathbb C)$.  Moreover, $Y$ is irreducible and
$U$ is open in $Y$, so we deduce that $U$ is also irreducible and in particular $U(\mathbb C)$
is path connected.  Hence the result is independent of the choice of base point, and we will omit
any reference to it in the rest of the proof.

If $n=0$, then it is enough to see 
that $\mathbb C _{R}(\iota _{U,Y})$ induces a bijection on $\pi _{0}$.  But this is
clear since $Y$ is irreducible.
	
Now we assume that $n>0$, so $d\geq 2$.  By Lemma \ref{lem.qproj.stepc}, $\mathbb C _{R}(\iota _{U,Y})$ induces
a bijection on the fundamental groups $\pi _{1}$. Hence, it only remains to check that  $\pi_{i}(\mathbb C _{R}(\iota _{U,Y}))$
is an isomorphism for $2\leq i\leq 2n$.

Let $\widetilde{Y(\mathbb C)}$ be the universal covering space of $Y(\mathbb C)$, and consider the following cartesian square of
topological spaces:
	\begin{align*}
		\xymatrix{p_{Y}^{-1}(U(\mathbb C)) \ar[d]_-{p_{U}} \ar[rr]^-{\widetilde{\mathbb C _{R}(\iota _{U,Y})}}&& \widetilde{Y(\mathbb C)}\ar[d]^-{p_{Y}}\\
		U(\mathbb C) \ar[rr]_-{\mathbb C _{R}(\iota _{U,Y})}&& Y(\mathbb C)}
	\end{align*}
Notice that $\widetilde{Y(\mathbb C)}$ has the structure of an analytic manifold 
where $p_{Y}$ is a local analytic isomorphism.  Since $Y(\mathbb C)$
is smooth we deduce that $\widetilde{Y(\mathbb C)}$ is irreducible as an analytic manifold. Thus, $p_{Y}^{-1}(U(\mathbb C))$ is irreducible and in
particular path connected,
since it is the open complement of the closed analytic submanifold $p_{Y}^{-1}(Z(\mathbb C))$ of $\widetilde{Y(\mathbb C)}$, where
$Z=Y\backslash U$.  Hence, $p_{U}$ is a covering map.

We claim that $p_{Y}^{-1}(U(\mathbb C))$ is the universal covering space of $U(\mathbb C)$.  In effect, it suffices to show that
$\pi_{1}(p_{Y}^{-1}(U(\mathbb C)))=\ast$.  Now, we observe that $\pi _{1}(p_{U})$ is injective since $p_{U}$ is a covering; and by
Lemma \ref{lem.qproj.stepc}, $\pi_{1}(\mathbb C _{R}(\iota _{U,Y}))$ is an isomorphism.  Thus, by the commutativity of the diagram
above we conclude that $\pi_{1}(p_{Y})\circ \pi_{1}(\widetilde{\mathbb C _{R}(\iota _{U,Y})})$ is also injective, but this map
is trivial since it factors through $\pi_{1}(\widetilde{Y(\mathbb C)})=\ast$.  Thus, $\pi_{1}(p_{Y}^{-1}(U(\mathbb C)))=\ast$, as we wanted.

Hence, in order to show that $\pi_{i}(\mathbb C _{R}(\iota _{U,Y}))$ is an isomorphism for $2\leq i\leq 2n$, it suffices to show
that $\pi_{i}(\widetilde{\mathbb C _{R}(\iota _{U,Y})})$ is an isomorphism for $2\leq i\leq 2n$.
But this follows directly from Corollary \ref{cor.qproj.stepd} and 
\cite[Prop. 1.5.2]{MR1944041}, since $\widetilde{\mathbb C _{R}(\iota _{U,Y})}$ is an open immersion between
analytic manifolds with smooth closed complement $p_{Y}^{-1}(Z(\mathbb C))$ of complex codimension $\geq n+1$.
\end{proof}

\end{section}

%% file: sect4_unstsf.tex
\begin{section}{Some Computations}
		\label{sect-3}

	In this section we study the behavior of unstable and stable slices with respect to the Quillen
	adjunction between the category of pointed simplicial presheaves on $Sm_{X}$ and the category of symmetric
	$T$-spectra on $\mathcal M$:
		\begin{align} 
				\label{inf.susp.adj}
			(\Sigma _{T}^{\infty}, ev_{0},\varphi):\unsmot \rightarrow \TspectraX .
		\end{align}
	We also carry out some computations of unstable slices using 
	a method similar to \cite[\S 4]{Pelaez:2011fk}. 
	
\begin{defi}
		\label{def.uns.nslice.subclass}
	Let $A$ in $\unstablehomotopyX$ be a pointed simplicial presheaf on $Sm_{X}$, and $n\geq 0$ an arbitrary integer.  We will
	say that $A$ is an \emph{$n$-slice} if $A\cong s_{n}A$ in $\unstablehomotopyX$, i.e. the maps $\tau _{\leq n+1}^{A}$ and $i_{n}^{A}$ in
	the following diagram are isomorphisms in $\unstablehomotopyX$
		\[	\xymatrix{&A\ar[d]^-{\tau _{\leq n+1}^{A}}&\\
					s_{n}A \ar[r]_-{i_{n}^{A}}& t_{n+1}A \ar[r]_-{\nu _{n}^{A}}& t_{n}A .}
		\]
\end{defi}	

Recall that the functors, $t_{n}:\unstablehomotopyX \rightarrow \unstablehomotopyX$ were defined 
in Proposition \ref{prop.slices=>homotopy}.

\begin{lem}
		\label{lemma.computingslices-via-birinvs.2}
	Let $n\geq 0$ be an arbitrary integer and
	$f:A\rightarrow B$ be a map in $\unsmot$ which is a weak equivalence in $\nwbiratunsmotX$.
	Then, for every integer $1\leq j \leq n+1$,
	$t_{j}(f)$ is an isomorphism in $\unstablehomotopyX$.
\end{lem}
\begin{proof}
	By Lemma \ref{lemma.weakequivalences.under.tower} we deduce that $f$ is a weak equivalence in 
	$\jminonewbiratunsmotX$.  By Lemma \ref{lem.slices.preserve.wequivs}\eqref{lem.slices.preserve.wequivs.a},
	we conclude that $\tau _{j}(f)$ is a weak equivalence in $\unsmot$.
	Hence, the result follows from Proposition
	\ref{prop.slices=>homotopy}(\ref{prop.slices=>homotopy.a}).
\end{proof}
	
\begin{lem}
		\label{lemma.computingslices-via-birinvs}
	Let $n\geq 1$ be an arbitrary integer and
	$f:A\rightarrow B$ be a map in $\unstablehomotopyX$ such that
	$t_{n+1}(f)$ and $t_{n}(f)$ are isomorphisms in $\unstablehomotopyX$.
	Then the unstable $n$-slice of $f$, $s_{n}(f)$ is also an isomorphism in $\unstablehomotopyX$.
\end{lem}
\begin{proof}
	It follows from Proposition \ref{slices.fit.fibreseqs}
	that the rows in the following commutative diagram
	are fibre sequences in $\unstablehomotopyX$
		\[	\xymatrix{s_{n}(A) \ar[d]_-{s_{n}(f)} \ar[r] & t_{n+1}(A)\ar[d]^-{t_{n+1}(f)} \ar[r] & t_{n}(A) 
							\ar[d]^-{t_{n}(f)}  \\
							s_{n}(B) \ar[r]  & t_{n+1}(B) \ar[r] & t_{n}(B)}
		\]
	Thus, by \cite[I.3 Prop. 5(iii)]{MR0223432}
	we conclude that $s_{n}(f)$ is also an isomorphism in $\unstablehomotopyX$.
\end{proof}

\begin{thm}
	\label{thm.class.uns.slices}
Let $A$ in $\unstablehomotopyX$ and $n\geq 0$.  Then $A$ is an $n$-slice 
(see Definition \ref{def.uns.nslice.subclass})
if and only if the following conditions hold:
	\begin{enumerate}
		\item \label{thm.class.uns.slices.a}  $t_{j}A\rightarrow \ast$ is an isomorphism in $\unstablehomotopyX$ 
		for every $1\leq j\leq n$,
		\item \label{thm.class.uns.slices.b} the natural map $\tau ^{A}_{\leq n+1}:A\rightarrow t_{n+1}A$
				is an isomorphism in $\unstablehomotopyX$ (see Theorem \ref{thm.slice.tower.derived}).
	\end{enumerate}
\end{thm}
\begin{proof}
($\Rightarrow$):  Assume that $A$ is an $n$-slice.    Thus, the maps $\tau _{\leq n+1}^{A}$, $i_{n}^{A}$ in the diagram below
are isomorphisms in $\unstablehomotopyX$ (see Definition \ref{def.uns.nslice.subclass}).  Thus, 	the map $\nu _{n}^{A}$ is null homotopic in 
$\unstablehomotopyX$ (i.e. it factors through $\ast$), since the horizontal row in the following diagram
is a fibre sequence in  $\unstablehomotopyX$ by Proposition \ref{slices.fit.fibreseqs}:
	\[	\xymatrix{&A\ar[d]^-{\tau _{\leq n+1}^{A}}&\\
					s_{n}A \ar[r]_-{i_{n}^{A}}& t_{n+1}A \ar[r]_-{\nu _{n}^{A}}& t_{n}A .}
	\]
Hence, the following diagram in $\unstablehomotopyX$ commutes
	\[ \xymatrix{t_{n+1}A \ar[r]^-{\nu _{n}^{A}} \ar[dr]& t_{n}A \\
					& \ast \ar[u]}
	\]
Applying the functor $t_{j}$, we obtain the following commutative diagram in $\unstablehomotopyX$
	\[ \xymatrix{t_{j}A\cong t_{j}(t_{n+1}A) \ar[rr]^-{t_{j}(\nu _{n}^{A})} \ar[drr]&& t_{j}(t_{n}A)\cong t_{j}A \\
					&& t_{j}(\ast)=\ast \ar[u]}
	\]
where $t_{j}A\cong t_{j}(t_{n+1}A)$ and $t_{j}(t_{n}A)\cong t_{j}A$ are isomorphic in $\unstablehomotopyX$ 
by Corollary \ref{cor.tq.respect.order}.  Now, Proposition \ref{tq.respect.order} implies that $t_{j}(\nu _{n}^{A})$ 
is an isomorphism in $\unstablehomotopyX$ since $n\geq j$.  
Thus, we conclude that $t_{j}A$ is a retract of $\ast$
in $\unstablehomotopyX$, so $t_{j}A\cong \ast$ in $\unstablehomotopyX$.  

($\Leftarrow$):  Assume that \eqref{thm.class.uns.slices.a} and \eqref{thm.class.uns.slices.b} hold.  Then,
by Proposition \ref{slices.fit.fibreseqs} we deduce that $s_{n}A\cong t_{n+1}A\cong A$ in $\unstablehomotopyX$ .
\end{proof}

Now, we will study the relation between the stable slice filtration (see Theorems
\ref{thm.modelstructures-slicefiltration},
\ref{thm.orthogonal=wbirationalTspectra}) and its unstable analogue (see Theorem
\ref{thm.slice.tower.derived}).    Consider the infinite suspension of the set of maps described
in Definition  \ref{def.localizing-maps2}:
	\begin{align}
			\label{infsusp.nwbirat}
		\infsusp WB_{n}=\{ \infsusp & \iota _{U,Y} \; |\; \iota _{U,Y}: U_{+}\rightarrow Y_{+} \text{ open immersion;} \\
									& Y, Z=Y\backslash U \in Sm_{X}; Y \text{ irreducible}; (codim_{Y}Z)\geq n+1 \} \nonumber
	\end{align}

\begin{prop}
	\label{prop.slice.infsusp.bhv}
Let $n\geq 0$ be an arbitrary integer.  In the following commutative diagram, all the arrows are left Quillen functors:
	\begin{align*}
		\xymatrix{\unsmot \ar[r]^-{\infsusp} \ar[d]_-{id} & \TspectraX \ar[d]^-{id}\\
					\nwbiratunsmotX \ar[r]_-{\infsusp} & \nwbiratTspectraX }
	\end{align*}
\end{prop}
\begin{proof}
It is clear that the diagram commutes and that the top horizontal arrow is a left Quillen functor.
The vertical arrows are left Quillen functors since by construction 
$\nwbiratunsmotX$ (resp. $\nwbiratTspectraX$) is a left Bousfield localization of $\unsmot$ (resp. $\TspectraX$).  

Now, we observe that all the maps in \eqref{infsusp.nwbirat} are contained in \eqref{infsusp.nwbirat.locmaps}.  Hence, we deduce that
all the maps in \eqref{infsusp.nwbirat} are weak equivalences in $\nwbiratTspectraX$, since by construction it is the left Bousfield
localization of $\TspectraX$ with respect to the maps in \eqref{infsusp.nwbirat.locmaps}.
Finally, by the universal property of left Bousfield localizations 
(see \cite[3.3.19.(1) and 3.1.1.(1)]{MR1944041}) we conclude that the bottom horizontal arrow is also a left Quillen functor.
\end{proof}

We will write $\nwbiratstablehomotopy$ for the homotopy category of $\nwbiratTspectraX$, and
	\begin{align*} 
		(\Sigma _{T}^{\infty}, \Omega _{T}^{\infty},\varphi):\unstablehomotopyX 
			\rightarrow \stablehomotopyX
	\end{align*}
for the derived adjunction of \eqref{inf.susp.adj}.

Let $E\in \TspectraX$ be an arbitrary symmetric $T$-spectrum, and $n\geq 0$ an arbitrary integer.  Consider the functors
$s_{n}:\stablehomotopyX \rightarrow \stablehomotopyX$, $s_{<n+1}:\stablehomotopyX \rightarrow \stablehomotopyX$ as defined in
Theorem \ref{thm.modelstructures-slicefiltration}.

\begin{prop}
	\label{slice.orth.fibran.del}
$\infdel s_{n}E$ and
$\infdel s_{<n+1}E$  are fibrant in $\nwbiratunsmotX$.
\end{prop}
\begin{proof}
By Theorems \ref{thm.modelstructures-slicefiltration}\eqref{thm.modelstructures-slicefiltration.b}, 
\ref{thm.orthogonal=wbirationalTspectra} (resp.\,\ref{thm.modelstructures-slicefiltration}\eqref{thm.modelstructures-slicefiltration.b}-\eqref{thm.modelstructures-slicefiltration.c}, 
\ref{thm.orthogonal=wbirationalTspectra})
we can assume that $\infdel s_{<n+1}E$ (resp.\,$\infdel s_{n}E$) is fibrant in $\nwbiratTspectraX$.
Hence, the result follows from Proposition \ref{prop.slice.infsusp.bhv}.
\end{proof}

On the other hand, combining \eqref{orth.dist.triang} with Theorems
\ref{thm.modelstructures-slicefiltration}\eqref{thm.modelstructures-slicefiltration.b}, 
\ref{thm.orthogonal=wbirationalTspectra} we deduce that:
	\begin{align*}
	\xymatrix{f_{n+1}E \ar[r]& E \ar[r]& s_{<n+1}E}
	\end{align*}
is a fibre sequence (and a distinguished triangle) in $\nwbiratstablehomotopy$ and $\stablehomotopyX$.
Similarly, combining \eqref{eq.def.slices} with Theorems \ref{thm.modelstructures-slicefiltration}\eqref{thm.modelstructures-slicefiltration.b}-\eqref{thm.modelstructures-slicefiltration.c}, 
\ref{thm.orthogonal=wbirationalTspectra} we deduce that:
	\begin{align*}
	\xymatrix{s_{n}E \ar[r]& s_{<n+1}E \ar[r]& s_{<n}E }
	\end{align*}
is a fibre sequence (and a distinguished triangle) in $\nwbiratstablehomotopy$ and $\stablehomotopyX$.
Therefore, by Proposition \ref{prop.slice.infsusp.bhv} we conclude:

\begin{prop}
	\label{prop.inf.susp.comp.fibr}
Let $E\in \TspectraX$ be an arbitrary symmetric $T$-spectrum, and $n\geq 0$ an arbitrary integer.
Then the following are fibre sequences in $\homotnwbiratunsmotX$ and $\unstablehomotopyX$:
	\begin{align*}
		\infdel f_{n+1}E \rightarrow \infdel E \rightarrow \infdel s_{<n+1}E &\\
		\infdel s_{n}E \rightarrow \infdel s_{<n+1}E \rightarrow \infdel s_{<n}E & \; .
	\end{align*}
\end{prop}

By looking at Propositions \ref{slice.orth.fibran.del} and \ref{prop.inf.susp.comp.fibr}, it is natural to ask
if $\infdel s_{n}E$ is an unstable $n$-slice (see Definition \ref{def.uns.nslice.subclass}).  If $n=0$,
then by Theorem \ref{thm.class.uns.slices} we conclude that $\infdel s_{0}E$ is an unstable zero-slice.
However, if $n\geq 1$ then $\infdel s_{n}E$ is not an unstable $n$-slice in general.

\begin{exam}
	\label{infdel.not.comp.slice}
Let the base scheme $X=\mathrm{Spec}\; k$, with $k$ a perfect field and $\ell \geq 3$ a prime different from
the characteristic of $k$.  Consider the spectrum
$\mathbf H _{\ell}$ which  represents in $\stablehomotopyX$ motivic cohomology 
with $\mathbb Z / \ell$-coefficients \cite[\S 6.1]{MR1648048}.  Then, $\infdel s_{\ell -1}
(S^{2-\ell}\wedge \Gm ^{\ell -1}\wedge \mathbf H _{\ell})$
is not an unstable $\ell -1$-slice.

In effect,
Guillou and Weibel construct a non-trivial
map in $\unstablehomotopyX$ (taking r=1 in \cite[Ex. 8.1]{GuiWei} and using the grading for motivic cohomology
introduced in \cite[\S 6, p. 595]{MR1648048}):
	\begin{align*}
		v: \infdel (S^{2-\ell}\wedge \Gm ^{\ell -1}\wedge \mathbf H _{\ell})
		\rightarrow \infdel (\Gm \wedge \mathbf H _{\ell})
	\end{align*}
On the other hand, by the work of Levine \cite[Lem. 10.4.1]{MR2365658} $s_{0}\mathbf H _{\ell}\cong 
\mathbf H _{\ell}$ in $\stablehomotopyX$.  Thus, we conclude that in $\stablehomotopyX$:
	\begin{align*}
		s_{\ell -1}(S^{2-\ell}\wedge \Gm ^{\ell -1}\wedge \mathbf H _{\ell}) & \cong 
			S^{2-\ell}\wedge \Gm ^{\ell -1}\wedge \mathbf H _{\ell}\\
		s_{1}(\Gm \wedge \mathbf H _{\ell}) & \cong \Gm \wedge \mathbf H _{\ell}.
	\end{align*}
Now, by Theorem
\ref{thm.slice.tower.derived}\eqref{thm.slice.tower.derived.c}:
	\begin{align*}
		\tau _{\leq 2}^{\infdel s_{\ell -1}(S^{2-\ell}\wedge \Gm ^{\ell -1}\wedge \mathbf H _{\ell})}:
		\infdel s_{\ell -1}(S^{2-\ell}\wedge \Gm ^{\ell -1}\wedge \mathbf H _{\ell})\rightarrow
		t_{2}(\infdel s_{\ell -1}(S^{2-\ell}\wedge \Gm ^{\ell -1}\wedge \mathbf H _{\ell}))
	\end{align*}
is a weak equivalence in $\onewbiratunsmotX$.  Combining
the adjunctions in Theorem \ref{tower.Quillenfunctors}\eqref{tower.Quillenfunctors.b} with the
fact that $\infdel s_{1}(\Gm \wedge \mathbf H _{\ell})\cong \infdel (\Gm \wedge \mathbf H _{\ell})$ is
fibrant in $\onewbiratunsmotX$ (by Proposition \ref{slice.orth.fibran.del}) and the non-triviality of $v$,
we deduce that 
	\begin{align*}
		t_{2}(\infdel s_{\ell -1}(S^{2-\ell}\wedge \Gm ^{\ell -1}\wedge \mathbf H _{\ell}))\rightarrow \ast
	\end{align*}
is not an isomorphism in $\unstablehomotopyX$.  Hence, by Theorem
\ref{thm.class.uns.slices}\eqref{thm.class.uns.slices.a} it follows that
$\infdel s_{\ell -1}
(S^{2-\ell}\wedge \Gm ^{\ell -1}\wedge \mathbf H _{\ell})$
is not an unstable $\ell -1$-slice.
\end{exam}

\begin{rmk}
	\label{uns.slic.finer}
From Example \ref{infdel.not.comp.slice} we observe that the unstable slices are non-trivial
and finer invariants than the corresponding stable slices.
\end{rmk}

Now we proceed to describe some computations of unstable slices. 

\begin{rmk}[see Remark \ref{rmk.evenPosttow.alggeom}]
	\label{rmk.conn.spheres}
The following result is an unstable motivic analogue of the standard result in topology:
$\pi _{i}(S^{2n})=\ast$ for $i\leq 2n-1$, i.e. the topological connectivity of the spheres.
\end{rmk}

We will write $T^{r}$ for $S^{r}\wedge \Gm ^{r}$, where $r\geq 0$ is an integer.  

\begin{thm}
		\label{thm.vanishing.effective.cats}
	Let $r\geq 1$; $1\leq m\leq r$; $0\leq n\leq r-1$ be arbitrary integers.  Then:
		\begin{enumerate}
			\item	\label{thm.vanishing.effective.cats.a}  For every $Z\in Sm_{X}$, the natural map 
					$t_{m}(T^{r}\wedge Z_{+})\rightarrow \ast$ is an isomorphism
					in $\unstablehomotopyX$.
			\item	\label{thm.vanishing.effective.cats.b}  For every $Z\in Sm_{X}$, the natural map 
					$s_{n}(T^{r}\wedge Z_{+})\rightarrow \ast$ is an isomorphism
					in $\unstablehomotopyX$.
		\end{enumerate}
\end{thm}
\begin{proof}
	By Lemma \ref{lemma.computingslices-via-birinvs} and the definition of the zero slice (see 
	Definition \ref{def.unstable.slice}) 
	it is enough to prove (\ref{thm.vanishing.effective.cats.a}).
	
	We claim that the natural map $T^{r}\wedge Z_{+}\rightarrow \ast$ is a weak equivalence
	in $\rminonewbiratunsmotX$.  In effect, \cite[Prop. 2.17(2)]{MR1813224} implies that the following diagram
	is a cofibre sequence in $\unsmot$
		\begin{align}
				\label{diagram.thm.vanishing.effective.cats}
			\xymatrix{(\mathbb A ^{r}_{Z}\backslash Z) \ar[r]^-{\alpha _{r,Z}}& 
							\mathbb A ^{r}_{Z} \ar[r]& T ^{r}\wedge Z_{+}}
		\end{align}
	where $Z$ is embedded in $\mathbb A ^{r}_{Z}$ via the zero section.  Now, the identity functor
	$id:\unsmot \rightarrow \rminonewbiratunsmotX$ is a left Quillen functor, since $\rminonewbiratunsmotX$ is a
	left Bousfield localization of $\unsmot$; hence, diagram (\ref{diagram.thm.vanishing.effective.cats})
	is also a cofibre sequence in $\rminonewbiratunsmotX$.  Therefore, 
	by  the dual of \cite[I.3 Prop. 5(iii)]{MR0223432} it is enough to show that
	$\alpha _{r,Z}$ is a weak equivalence in $\rminonewbiratunsmotX$.  But this follows from the definition
	of $\rminonewbiratunsmotX$ (see Definition \ref{def.localizationAmod-weaklyBirat}),
	since the codimension of the closed complement of the open immersion $\alpha _{r,Z}$ is $r$.
	
	Finally, the result follows from Lemma \ref{lemma.computingslices-via-birinvs.2}.
\end{proof}
	
	Given $Y\in Sm_{X}$, let  $\mathbb P ^{n}(Y)$ denote the trivial projective bundle of rank $n$ over $Y$.
	Consider the following filtered diagram in $\unsmot$
		\begin{equation}
				\label{eq.inf.proj.space}
			\mathbb P^{0}(Y)\rightarrow \mathbb P ^{1}(Y)\rightarrow \cdots \rightarrow 
			\mathbb P^{n}(Y)\rightarrow \cdots
		\end{equation}
	given by the inclusions of the respective hyperplanes at infinity.
	
\begin{rmk}[see Remark \ref{rmk.evenPosttow.alggeom}]
	\label{rmk.homgps.cp.projsp}
The following result is an unstable motivic analogue of the standard result in topology:
$\pi _{i}(\mathbb P ^{n}(\mathbb C))\cong \pi _{i}(\mathbb P ^{n+1}(\mathbb C))$ for $i\leq 2n$.
\end{rmk}
	
\begin{thm}
		\label{thm.slices.proj.space}
		\begin{enumerate}
			\item	\label{thm.slices.proj.space.a}  For any integer $1\leq j\leq n+1$, the diagram \eqref{eq.inf.proj.space} 
						induces the following isomorphisms in $\unstablehomotopyX$
							\[	\xymatrix@C=1.5pc@R=0.5pc{t_{j}(\mathbb P^{n}(Y)_{+}) \ar[r]^-{\cong} & 
										t_{j}(\mathbb P ^{n+1}(Y)_{+})
										\ar[r]^-{\cong} &
										 \cdots \ar[r]^-{\cong} & t_{j}(\mathbb P^{n+m}(Y)_{+})
										\ar[r]^-{\cong} &\cdots}
							\]
			\item	\label{thm.slices.proj.space.b}  For any integer $0\leq j\leq n$, the diagram \eqref{eq.inf.proj.space} 
						induces the following isomorphisms in $\unstablehomotopyX$
							\[	\xymatrix@C=1.5pc@R=0.5pc{s_{j}(\mathbb P^{n}(Y)_{+}) \ar[r]^-{\cong} & 
										s_{j}(\mathbb P ^{n+1}(Y)_{+})
										\ar[r]^-{\cong} & 
										 \cdots \ar[r]^-{\cong} & s_{j}(\mathbb P^{n+m}(Y)_{+})
										\ar[r]^-{\cong} &\cdots}
							\]
		\end{enumerate}
\end{thm}
\begin{proof}
	By Lemma \ref{lemma.computingslices-via-birinvs} and the definition of the zero slice (see 
	Definition \ref{def.unstable.slice}) 
	it is enough to prove (\ref{thm.vanishing.effective.cats.a}).
	
	Let $k>n$, and consider the closed embedding induced by the diagram (\ref{eq.inf.proj.space}),
	$\lambda ^{k}_{n}:\mathbb P^{n}(Y)\rightarrow \mathbb P^{k}(Y)$.  It is possible to choose a linear embedding
	$\mathbb P^{k-n-1}(Y)\rightarrow \mathbb P^{k}(Y)$ such that its open complement $U_{k,n}$ contains
	$\mathbb P^{n}(Y)$ (embedded via the zero section $\sigma _{n}^{k}$)
	and has the structure of a vector bundle over $\mathbb P^{n}(Y)$:
		\[ \xymatrix{ U_{k,n} \ar[r]^-{v^{k}_{n}} \ar@<-1ex>[d]& \mathbb P ^{k}(Y) & 
						\mathbb P ^{k-n-1}(Y) \ar[l]\\
						\mathbb P ^{n}(Y) \ar@<-1ex>[u]_-{\sigma ^{k}_{n}} \ar@/_1pc/[ur]_-{\lambda _{n}^{k}} & &}
		\]
	By homotopy invariance $t_{j}(\sigma ^{k}_{n})$ is an isomorphism in $\unstablehomotopyX$
	for every integer $j$.
	On the other hand, $v ^{k}_{n}$ is a weak equivalence in $\nwbiratunsmotX$ since
	the codimension of its closed complement is $n+1$
	(see Definition \ref{def.localizationAmod-weaklyBirat}).  Thus, Lemma \ref{lemma.computingslices-via-birinvs.2}
	implies that  if $1\leq j\leq n+1$, then $t_{j}(v ^{k}_{n})$ is also an isomorphism in $\unstablehomotopyX$.
	
	Therefore, $t_{j}(\lambda ^{k}_{n})=t_{j}(v^{k}_{n})\circ t_{j}(\sigma ^{k}_{n})$ 
	is an isomorphism in $\unstablehomotopyX$
	for $1\leq j\leq n+1$.
\end{proof}

\begin{defi}
		\label{def.Thom.space}
	Given a closed embedding $Z\rightarrow Y$ in $Sm_{X}$, let $N_{Y,Z}$ denote the associated
	normal bundle, and
	given a vector bundle $\pi: V\rightarrow Y$ with $Y\in Sm_{X}$, let $Th(V)$ denote the Thom
	space of $V$, i.e. the pushout of the following diagram in $\unsmot$
		\[	\xymatrix{V\backslash \sigma _{0}(Y)_{+} \ar[r] \ar[d]& V_{+} \ar[d]\\
					X \ar[r]& Th(V)}
		\]
	where  $\sigma _{0}:Y\rightarrow V$ is the zero section of $V$.
\end{defi}

\begin{rmk}[see Remark \ref{rmk.evenPosttow.alggeom}]
	\label{rmk.conn.Thomspc}
The following result is an unstable motivic analogue of the standard result in topology: 
If $\xi$ is a complex vector bundle of rank $r$ over a CW complex, then
$\pi_{k}(Th(\xi))=\ast$ for $k\leq 2r-1$ (see \cite[p.\,205 Lem.\,18.1]{MR0440554}).
\end{rmk}
	
\begin{thm}
		\label{thm.vanishing.normal.bundle}
	Let $n\geq 0$; $0\leq j \leq n$; $1\leq k\leq n+1$ be arbitrary integers, and
	let $\iota _{U,Y}\in WB_{n}$ with closed complement $Z$.   
	Then:
		\begin{enumerate}
			\item \label{thm.vanishing.normal.bundle.a}	The natural map $t_{k}(Th(N_{Y,Z}))\rightarrow \ast$
						is an isomorphism in $\unstablehomotopyX$.
			\item \label{thm.vanishing.normal.bundle.b}	The natural map $s_{j}(Th(N_{Y,Z}))\rightarrow \ast$
						is an isomorphism in $\unstablehomotopyX$.
		\end{enumerate}
\end{thm}
\begin{proof}
	By Lemma \ref{lemma.computingslices-via-birinvs} and the definition of the zero slice (see 
	Definition \ref{def.unstable.slice}) it is enough to prove (\ref{thm.vanishing.effective.cats.a}).
	
	Now, it follows from the Morel-Voevodsky homotopy purity theorem 
	(see \cite[Thm. 2.23]{MR1813224}) that
	the following diagram is a cofibre sequence in $\unsmot$
		\[	\xymatrix{U_{+} \ar[r]^-{\iota _{U,Y}}& Y_{+} \ar[r]& Th(N_{Y,Z})}
		\]
	and Theorem \ref{tower.Quillenfunctors} implies that the diagram above is also a cofibre sequence in
	$\nwbiratunsmotX$.  But, $\iota _{U,Y}$ is a weak equivalence in $\nwbiratunsmotX$, since
	it is defined as a left Bousfield localization of $\unsmot$ with respect to $WB_{n}$.
	Hence, the dual of \cite[I.3 Prop. 5(iii)]{MR0223432} implies that the map
		\[	Th(N_{Y,Z})\rightarrow \ast
		\]
	is a weak equivalence in $\nwbiratunsmotX$, and the result then follows from
	Lemma \ref{lemma.computingslices-via-birinvs.2}.
\end{proof}
	
\begin{rmk}[see Remark \ref{rmk.evenPosttow.alggeom}]
	\label{rmk.conn.Thomspc2}
The following result is an unstable motivic analogue of the standard result in topology: 
If $\xi$ is a complex vector bundle of rank $r$ over a CW complex $X$, then
$Th(\xi)$ has a cell decompostion given by the base point and $(n+2r)$-cells which are in bijection
with the $n$-cells of $X$, for $n>0$ (see \cite[p.\,205 Lem.\,18.1]{MR0440554}).
\end{rmk}	
	
	In the case $U=Y$, the following result is due to Morel-Voevodsky 
	(see Proposition 2.17(2) in 
	\cite{MR1813224}), and this particular case is an ingredient in our proof.
	
\begin{thm}
		\label{thm.slices-thom}
	Let $n\geq 0$; $0\leq j \leq n$; $1\leq k\leq n+1$ be arbitrary integers, and
	let $\iota _{U,Y}\in WB_{n}$.  Consider a vector bundle
	$\pi :V\rightarrow Y$ of rank $r$
	together with a trivialization $\lambda :\pi ^{-1}(U)\rightarrow \mathbb A ^{r}_{U}$ of its restriction to $U$.
	Then:
		\begin{enumerate}
			\item \label{thm.slices-thom.a}  There exists an isomorphism in $\unstablehomotopyX$
					\[	t_{k}(Th(V))\cong t_{k}(S^{r}\wedge \Gm ^{r}\wedge Y_{+})
					\]
			\item \label{thm.slices-thom.b}  There exists an isomorphism in $\unstablehomotopyX$
					\[	s_{j}(Th(V))\cong s_{j}(S^{r}\wedge \Gm ^{r}\wedge Y_{+})
					\]
		\end{enumerate}
\end{thm}
\begin{proof}
	Let $Z\in Sm_{X}$ be the closed complement of $\iota _{U,Y}$.
	Consider the following diagram in $Sm_{X}$, where all the squares are cartesian
		\[	\xymatrix{\pi ^{-1}(Z)\cap (V\backslash \sigma _{0}(Y)) \ar[r] \ar[d] & V
						\backslash \sigma _{0}(Y) 
						\ar[d] & \pi ^{-1}(U) \cap (V\backslash \sigma _{0}(Y)) \ar[l]_-{\beta} \ar[d] \\
						\pi ^{-1}(Z) \ar[r] \ar[d]& V \ar[d]^-{\pi} & \pi ^{-1}(U) \ar[l]_-{\alpha} \ar[d] \\
						Z \ar[r] & Y & U \ar[l]^-{\iota _{U,Y}}}
		\]
	and let $\gamma:Th(\pi ^{-1}(U))\rightarrow Th(V)$ be the induced map between the 
	corresponding Thom spaces.
	We observe that $\alpha, \beta$ also belong to $WB_{n}$; thus, $\alpha, \beta$
	are weak equivalences in $\nwbiratunsmotX$ (see Definition \ref{def.localizationAmod-weaklyBirat}).
	We claim that 
		\[ \gamma: Th(\pi^{-1}(U)) \rightarrow Th(V)
		\]
	is also a weak equivalence in $\nwbiratunsmotX$.  In effect,
	by construction of the Thom spaces, we deduce that 
	the rows in the following commutative diagram are 
	in fact cofibre sequences in $\unsmot$
		\[	\xymatrix{\pi ^{-1}(U)\cap 
			(V\backslash \sigma _{0}(Y))_{+} \ar[r] \ar[d]_-{\beta} & 
			\pi ^{-1}(U)_{+} \ar[r] \ar[d]_-{\alpha} & Th(\pi ^{-1}(U)) 
			\ar[d]_-{\gamma} \\
			(V\backslash \sigma _{0}(Y))_{+} \ar[r] & V_{+} \ar[r]
			& Th(V)}
		\]
	But, Theorem \ref{tower.Quillenfunctors} implies that these rows are also cofibre
	sequences in $\nwbiratunsmotX$.  Since $\alpha$, $\beta$ are weak equivalences in $\nwbiratunsmotX$,
	it follows from the dual of \cite[I.3 Prop. 5(iii)]{MR0223432} that
	$\gamma$ is also a weak equivalence in $\nwbiratunsmotX$.
	
	Now, we use the trivialization $\lambda$ to obtain the following commutative diagram in $Sm_{X}$ where the
	rows are isomorphisms
		\[	\xymatrix{\mathbb A ^{r}_{U}\backslash U  \ar[d] & 
						& \pi ^{-1}(U) \cap (V\backslash \sigma_{0}(Y))
						\ar[ll]_-{\tilde{\lambda}}^-{\cong}\ar[d]\\
						\mathbb A ^{r}_{U}  \ar[dr]& & \pi ^{-1}(U) \ar[dl]^-{\pi _{U}} 
						\ar[ll]_-{\lambda}^-{\cong}\\ & U &}
		\]
	The same argument as above, shows that there is a weak 
	equivalence in  $\nwbiratunsmotX$
		\[ \bar{\lambda}: Th(\pi ^{-1}(U))
			\rightarrow Th(\mathbb A ^{r}_{U})
		\]
	On the other hand, \cite[Prop. 2.17(2)]{MR1813224} implies that there is a weak equivalence
	$w:Th(\mathbb A ^{r}_{U}) \rightarrow S^{r}\wedge \Gm ^{r}\wedge U_{+}$
	in $\unsmot$.  Since $\nwbiratunsmotX$ is a left Bousfield localization of $\unsmot$, 
	it follows from \cite[Prop. 3.3.3(1)(a)]{MR1944041} that
	$w$ is also a weak equivalence in $\nwbiratunsmotX$.  Moreover, by construction
	$\iota _{U,Y}$ is a weak equivalence in $\nwbiratunsmotX$ 
	(see Definition \ref{def.localizationAmod-weaklyBirat}),
	since we are assuming that $\iota _{U,Y}\in WB_{n}$; and Ken Brown's lemma 
	(see \cite[Lem. 1.1.12]{MR1650134}) implies that
	$S^{r}\wedge \Gm ^{r}\wedge \iota _{U,Y}$ is also a weak equivalence in $\nwbiratunsmotX$.
	
	Thus, all the maps in the following diagram are weak equivalences in $\nwbiratunsmotX$
		\[	\xymatrix{Th(\pi ^{-1}(U)) \ar[r]^-{\bar{\lambda}} \ar[d]_-{\gamma}& 
							Th(\mathbb A ^{r}_{U}) \ar[r]^-{w} & 
							S^{r}\wedge \Gm ^{r}\wedge U_{+} \ar[d]^-{S^{r}\wedge \Gm ^{r}\wedge \iota _{U,Y}}\\
							Th(V)  &&  S^{r}\wedge \Gm ^{r}\wedge Y_{+}}
		\]
	Finally, the result follows from Lemmas \ref{lemma.computingslices-via-birinvs.2} and
	\ref{lemma.computingslices-via-birinvs}.
\end{proof}

\begin{cor}
		\label{cor.slices-thom}
	Assume that the base scheme $X=\mathrm{Spec}\; k$, with $k$ a perfect field.
	Let $n\geq 0$; $0\leq j \leq n$; $1\leq k\leq n+1$ be arbitrary integers, and
	let $\iota _{U,Y}\in B_{n}$.  Consider a vector bundle
	$\pi :V\rightarrow Y$ of rank $r$
	together with a trivialization $\lambda :\pi ^{-1}(U)\rightarrow \mathbb A ^{r}_{U}$ of its restriction to $U$.
	Then:
		\begin{enumerate}
			\item \label{cor.slices-thom.a}  There exists an isomorphism in $\unstablehomotopyX$
					\[	t_{k}(Th(V))\cong t_{k}(S^{r}\wedge \Gm ^{r}\wedge Y_{+})
					\]
			\item \label{cor.slices-thom.b}  There exists an isomorphism in $\unstablehomotopyX$
					\[	s_{j}(Th(V))\cong s_{j}(S^{r}\wedge \Gm ^{r}\wedge Y_{+})
					\]
		\end{enumerate}
\end{cor}
\begin{proof}
	By Proposition \ref{prop.genericsmoothness-Quillenequiv} we can use
	the same argument as in Theorem \ref{thm.slices-thom} to prove the result.
\end{proof}

	Given a closed embedding $Z\rightarrow Y$ of smooth schemes over $X$, let $\blowup _{Z}Y$
	denote the blowup of $Y$ with center in $Z$.
	
\begin{rmk}[see Remark \ref{rmk.evenPosttow.alggeom}]
	\label{rmk.conn.blowups}
Combining Theorems \ref{thm.Post.real.comp} and \ref{prop.Postnikov.agree} with the proof of Theorem
\ref{thm.blowups} we obtain the following result in the classical topological setting, which although not
surprising it appears not to be in the literature:  Let $Z\rightarrow Y$ be a closed embedding of smooth
quasiprojective complex algebraic varieties, and let $\blowup _{Z}Y$ denote the blowup of $Y$ with center
in $Z$.  Then the $k$-skeleton of $Y$ is a retract of the $k$-skeleton of $\blowup _{Z}Y$, where $k$
is the codimension of $Z$ in $Y$.
\end{rmk}	

\begin{thm}
		\label{thm.blowups}
	Let $n\geq 1$; $0\leq j \leq n$; $1\leq k\leq n+1$ be arbitrary integers, and
	let $\iota _{U,Y}\in WB_{n}$ with closed complement $Z$.  
	Consider the following cartesian square in $Sm_{X}$
		\begin{equation}
				\label{eq.blowup}
			\begin{array}{c}
			\xymatrix{ D \ar[d]_-{q} \ar[r]^-{d}& \blowup _{Z}Y \ar[d]^-{p} & U \ar[l]_-{u} 
						\ar@{=}[d] \\
						 Z \ar[r]_-{i} & Y & U \ar[l]^-{\iota _{U,Y}}}
			\end{array}
		\end{equation}
	Then: 
		\begin{enumerate}
			\item	\label{thm.blowups.a}	$t_{1}(u)=s_{0}(u)$ is an isomorphism in 
				$\unstablehomotopyX$.
			\item	\label{thm.blowups.b}	$t_{k}(\iota _{U,Y})$ 
				is an isomorphism in $\unstablehomotopyX$, and the following diagram
				commutes
					\begin{align}
							\label{diagram.thm.blowups.b}
						\begin{array}{c}
							\xymatrix{& \blowup _{Z}Y_{+} \ar[d]^-{t_{k}(p)}\\
								Y_{+} \ar@{=}[r] \ar[ur]^-{\tilde{r}_{k}}& Y_{+}}
						\end{array}
					\end{align}
				where $\tilde{r}_{k}=t_{k}(u)\circ (t_{k}(\iota _{U,Y}))^{-1}$.		
			\item	\label{thm.blowups.c}	$s_{j}(\iota _{U,Y})$ 
				is an isomorphism in $\unstablehomotopyX$, and the following diagram
				commutes
					\begin{align}
							\label{diagram.thm.blowups.c}
						\begin{array}{c}
							\xymatrix{& \blowup _{Z}Y_{+} \ar[d]^-{s_{j}(p)}\\
								Y_{+} \ar@{=}[r] \ar[ur]^-{r_{j}}& Y_{+}}
						\end{array}
					\end{align}
				where $r_{j}=s_{j}(u)\circ (s_{j}(\iota _{U,Y}))^{-1}$.
		\end{enumerate}
\end{thm}
\begin{proof}
	(\ref{thm.blowups.a}):  We observe that $u\in WB_{0}$ (see Definition \ref{def.localizing-maps2}).  
	Thus, $u$ is a weak equivalence in $\zerowbiratunsmotX$ 
	(see Definition \ref{def.localizationAmod-weaklyBirat});
	and the result follows from Lemma \ref{lemma.computingslices-via-birinvs.2}.
	
	(\ref{thm.blowups.b}): We observe that $\iota _{U,Y}$ is a weak equivalence in $\nwbiratunsmotX$
	(see Definition \ref{def.localizationAmod-weaklyBirat}).  
	Then Lemma \ref{lemma.computingslices-via-birinvs.2} implies that
	$t_{k}(\iota _{U,Y})$ is an isomorphism in $\unstablehomotopyX$.
	Thus, $\tilde{r}_{k}$ is well defined, and by functoriality the diagram (\ref{diagram.thm.blowups.b})
	commutes in $\unstablehomotopyX$.
		
	(\ref{thm.blowups.c}):  The result follows from (\ref{thm.blowups.b}) above together with Lemma
	\ref{lemma.computingslices-via-birinvs} and the definition of the zero slice
	(see Definition \ref{def.unstable.slice}).
\end{proof}

\end{section}